\documentclass[journal]{IEEEtran}
\onecolumn
\usepackage{mathrsfs}
\usepackage{amsmath}
\usepackage{amsthm}
\usepackage{graphicx}
\usepackage{amssymb}
\usepackage{epstopdf}
\usepackage{enumerate}
\usepackage{longtable,tabularx,float}
\usepackage{cite}
\usepackage{mathcomp}
\usepackage{supertabular}
\usepackage{stmaryrd}
\usepackage{color}
\usepackage{url}
\usepackage{makecell}
\usepackage[OT2,OT1]{fontenc}
\usepackage{bm}
\usepackage{bbm}
\usepackage{caption}
\newtheorem{corollary}{Corollary}[section]

\newtheorem{lemma}{Lemma}[section]
\newtheorem{theorem}{Theorem}[section]

\newtheorem{proposition}{Proposition}[section]
\newtheorem{example}{Example}[section]
\newtheorem{construction}{Construction}[section]
\newtheorem{fact}{Fact}[section]
\newtheorem{remark}{Remark}[section]

\newcommand{\C}{\mathcal {C}}

\renewcommand{\P}{\mathcal {P}}

\renewcommand{\S}{\mathcal {S}}
\newcommand{\Z}{\mathcal {Z}}
\newcommand{\M}{\mathcal {M}}
\renewcommand{\H}{\mathcal {H}}
\newcommand{\bbZ}{{\mathbb Z}}

\begin{document}

\title{Optimal Ternary Codes with Weight $w$ and Distance $2w-2$ in $\ell_1$-Metric}

\author{Xin~Wei,~Tingting~Chen,~and~Xiande~Zhang
        \thanks{X. Wei ({\tt weixinma@mail.ustc.edu.cn}) and T. Chen ({\tt ttchenxu@mail.ustc.edu.cn}) are with School of Mathematical Sciences and
        School of Cyber Security, University of Science and Technology of China, Hefei, 230026, Anhui, China.}

\thanks{X. Zhang ({\tt drzhangx@ustc.edu.cn}) is with School of Mathematical Sciences,
University of Science and Technology of China, Hefei, 230026, Anhui, China.  The research of X. Zhang is supported by NSFC under grant 11771419, and by ``the  Fundamental
Research Funds for the Central Universities''.}

}
\maketitle

\begin{abstract} The study of constant-weight codes in $\ell_1$-metric was motivated by the duplication-correcting problem for data storage in live DNA. It is interesting to determine the maximum size of a code given the length $n$, weight $w$, minimum distance $d$ and the alphabet size $q$. In this paper,  based on graph decompositions, we determine the maximum size of ternary codes with constant weight $w$ and distance $2w-2$ for all sufficiently large length $n$. Previously, this was known only for a very sparse family $n$ of density $4/w(w-1)$.

\end{abstract}


\section{Introduction}
\IEEEPARstart{C}{odes} in $\ell_1$-metric distance have attracted a lot attention in the last decade for their useful applications in rank-modulation scheme for flash memory \cite{barg2010codes,jiang2009error,tallini20111,zhou2014systematic,farnoud2013error,kabatianskycodes}. In the recent error correcting problem of data storage in live DNA,
 $\ell_1$-metric was also revealed to play an important role in attacking the tandem duplication errors  \cite{jain2017duplication}. To store information in live DNA, the reliability of data is crucial. A tandem duplication error, which means an error happens by inserting a copy of a segment of the DNA adjacent to its original position during replication, is one of the errors we must fight against during DNA coding.  According to \cite{lander2001initial}, tandem duplications constitute about 3\% of our human genome and may cause serious phenomena and severe information loss. As a consequence, duplication correcting codes have been studied by many recent works, for example \cite{jain2017duplication,kovavcevic2018asymptotically,lenz2018bounds,tang2019single,yehezkeally2019reconstruction}. Specially, the work in \cite{jain2017duplication} connects duplication correcting codes with constant-weight codes (CWCs) in $\ell_1$-metric over integers, and shows that the existence of optimal CWCs with certain weights and lengths will result in optimal tandem duplication correcting codes \cite[Theorem 20]{jain2017duplication}.

%  Specially, \cite{jain2017duplication} shows that given a non-negative integer $k$, a code can correct tandem duplication of length $k$ if and only if the zero signatures of the $z$-part of all $k$-congruent codewords form a constant-weight code (CWC) in $\ell_1$-metric on integers. If each segment of length $k$ is duplicated no more than $q-1$ times, then the zero signatures are $q$-ary. For more details we point the reader to \cite[Theorem 20]{jain2017duplication}.
%CWCs play an important role in coding theory \cite{MacwilliamsThe}.
Motivated by this connection, it is interesting to study CWCs in $\ell_1$-metric. The central problem regarding CWCs is to determine the exact value of maximum cardinality of a code with given length, weight, minimum distance, and the alphabet size. CWCs with Hamming distance have been well studied \cite{graham1980lower,brouwer1990new,agrell2000upper,zhang2012optimal,chee2014complexity,chee2015hanani,
chee2017linear,chee2017constructions,chee2019decompositions,Cao2007ConstructionsFG,Zhang2010Optimal,Bogdanova2000New} due to its wide applications in coding for bandwidth-efficient channels \cite{costello2007channel}, the design of oligonucleotide sequences for DNA computing \cite{king2003bounds,milenkovic2005design}, and so on.
If the code is binary, the  Hamming distance is indeed the $\ell_1$-metric. However, when it is $q$-ary with $q\geq 3$, the metrics are different.  The study of non-binary CWCs with $\ell_1$-metric is very rare. Based on our knowledge, only works in  \cite{vincent2018code,jinushi1990construction,chen2020optimal} are involved.

%When binary CWCs (i.e. $q=2$) have been extensively studied by many authors, there are also some papers dealing with nonbinary CWCs. In those papers, most of them are about Hamming distance, and results about $q$-ary CWCs on other distances are very few \cite{Bogdanova2000New,Cao2007ConstructionsFG,Zhang2010Optimal,zhang2012optimal}.

In this paper, we consider ternary CWCs with $\ell_1$-metric. A ternary code of constant $\ell_1$-weight $w$, minimum $\ell_1$-distance $d$ and  length $n$ is denoted by an $(n, d, w)_3$ code. The maximum size among all $(n, d, w)_3$ codes is denoted by $A_3(n, d, w)$, and a code achieving this size is called optimal. Optimal $(n, d, w)_3$ codes were constructed in \cite{chen2020optimal} when $w=3,4$  for all possible distances $d$ and length $n$ with only finite many cases undetermined. For general weight $w$ and distance $2w-2$, the authors in \cite{chen2020optimal} provided an upper bound  $A_3(n, 2w-2, w)\le\left\lfloor\frac{n(n-1-(w-1)(w-2))}{w(w-1)}\right\rfloor+n$, and showed that this bound can be achieved for $n$ is large enough and $n\equiv w, -2w+3, 1$ or $-w+2 \mod (w-1)w$. This is only a fraction of the numbers $n$ when $w$ is big.

Our main contribution in this paper is showing that, given the weight $w$, the upper bound of $A_3(n, 2w-w, w)$ provided in \cite{chen2020optimal} can be achieved for all sufficiently large $n$, which greatly improves the result in \cite{chen2020optimal}. We state our main result as follows.

\vspace{0.2cm}
\noindent\textbf{Main Result:}
For any integer $w\ge 5$, $A_3(n, 2w-2, w)=\left\lfloor\frac{n(n-1-(w-1)(w-2))}{w(w-1)}\right\rfloor+n$ for all sufficiently large $n$.
\vspace{0.2cm}

Further, we show that there exists an optimal $(n, 2w-2, w)_3$ code that is  balanced unless when
 $w$ is odd and $n\equiv1\mod (w-1)$ with certain restrictions. Here, a code is  balanced  means that the collection of supports of all codewords, viewed as cliques, forms a decomposition of the complete graph $K_n$. Our methods rely on a famous result of Alon et al. \cite{alon1998packing} on graph decompositions, which reduces the whole problem to looking for a sparse code with desired properties.

 This paper is organized as follows. In Section \ref{pre}, we first introduce some necessary definitions and notations, then connect  codes with graph packings to reduce the problem of finding an optimal code $\C$ to finding a proper sub-code $\S\subset \C$. Section \ref{mainconsturctionfor01} gives the general idea of how to find the sub-code $\S$, and illustrates the idea by taking  $n\equiv 0,1 \mod (w-1)$ with certain restrictions. A complete solution to the cases of  $n\equiv 0,1 \mod (w-1)$  is given in Section \ref{additionalconstructionA}, where a nonexistence result of an optimal balanced code with  $n\equiv1\mod (w-1)$ under additional conditions  is deduced. In Section \ref{Constructionforgeneralt}, we deal with all other cases, $n\equiv t \mod (w-1)$ with $2\leq t\leq w-2$ by an algorithmic construction of the required sub-code $\S$. Finally we conclude our results in Section \ref{conc.}.

\section{Preliminaries}\label{pre}

For integers $m\le m',$ we denote $[m, m']\triangleq\{m, m+1, \dots, m'\}$ and $[m]\triangleq [1, m]$ for short.
Let $q\ge 2$ be an integer.  A $q$-ary code $\mathcal C$ of length $n$ is a set of vectors in $I_q^n$, in which $I_q:=\{0, 1, \ldots, q-1\}$. An element in $\mathcal C$ is called a {\it codeword}. For two codewords $\bm u=(u_i)_{1\le i\le n}$ and  $\bm v=(v_i)_{1\le i\le n}$, the {\it $\ell_1$-distance} between $\bm u$ and $\bm v$ is $d(\bm u, \bm v)=\sum_{i=1}^n |u_i-v_i|$. Here the sum and subtraction are on $\bbZ$. The {\it $\ell_1$-weight} of $\bm v$ is the $\ell_1$-distance between $\bm v$ and vector $\bm 0$. If a code $\mathcal C$ satisfies that the $\ell_1$-weight of any codeword is $w$ for some integer $w$, then we call $\mathcal C$ a {\it constant-weight} code. If further the minimum $\ell_1$-distance between any two distinct codewords in $\mathcal C$ is at least $d$, we say that $\mathcal C$ is an $(n, d, w)_q$ code.
%, or an $(n, d, w)$ code if $\mathcal C$ is over nonnegative integers.

In this paper, we focus on ternary codes, that is $q=3.$  The maximum size among all $(n, d, w)_3$ codes is denoted by $A_3(n, d, w).$ We call an $(n, d, w)_3$ code $\mathcal C$ {\it optimal} if the size of $\mathcal C$ reaches $A_3(n, d, w).$ The following lemma from \cite{chen2020optimal} gives us an upper bound for $A_3(n, 2w-2, w).$

\begin{lemma}\label{upperbound}\cite{chen2020optimal}
$A_3(n, 2w-2, w)\le\left\lfloor\frac{n(n-1-(w-1)(w-2))}{w(w-1)}\right\rfloor+n.$
\end{lemma}

When weight $w$ is fixed, let $B(n)\triangleq\left\lfloor\frac{n(n-1-(w-1)(w-2))}{w(w-1)}\right\rfloor$ for convenience. Our main work is to show that the upper bound $B(n)+n$ can be achieved when $n$ is sufficiently large for any fixed $w$.

For any codeword $\bm u\in\mathcal C$, the {\it support} of $\bm u$ is defined as $supp(\bm u)=\{x\in[n], u_{x}\ne 0\}.$ There is a canonical one-to-one correspondence between a vector $\bm u$ in $I_3^n$ to a subset of $[n]\times[2]$. We can express the vector $\bm u=(u_1, u_2, \ldots, u_n)$ as a subset  $\{(x,{u_x})\in [n]\times[2]\mid {x\in supp(\bm u)}\}$, which we call the {\it labeled support} of $\bm u$. Specially for any $j\in[n]$, if $u_j=i\in I_3$, we say the position $j$ is {\it labeled} by the symbol $i$ in the codeword $\bm u$. For short, we usually write a member $(a, b)\in [n]\times[2]$ in the form $a_b$.

\begin{example}
When $n=5$,  codewords $12000,$ $10200,$ $11200$ and $20200$ can be described as $\{1_1, 2_2\},$ $\{1_1, 3_2\},$ $\{1_1, 2_1, 3_2\}$ and $\{1_2, 3_2\}$, respectively. Furthermore in $11200$, the position 3 is labeled by $2$.
\end{example}

For ternary codes of constant weight $w$, the \emph{type} of a codeword is denoted by $1^a2^b$ for some nonnegative integers $a$ and $b$ satisfying $a+2b=w$. Here, a codeword is of type $1^a2^b$ if there are in total $a$ different positions labeled by $1$ and $b$ different positions labeled by $2$ in it. We usually write $1^w$ instead of $1^w2^0$ for convenience. By a simple observation, there exists a necessary and sufficient condition of an $(n, 2w-2, w)_3$ code.

\begin{fact}\label{Transfer}
A ternary CWC code $\mathcal C$ with weight $w$ is an $(n, 2w-2, w)_3$ code if and only if the following conditions are satisfied:
\begin{itemize}
\item[(a)] For any two different codewords $\bm u$ and $\bm v$ in $\mathcal C$, $|supp(\bm u)\cap supp(\bm v)|\le 1$.
\item[(b)] If there exists a codeword $\bm u\in \mathcal C$ and a position $i\in[n]$ such that $u_i=2$, then for any other codeword $\bm v\in\mathcal C$, $v_i\ne 2$.
\end{itemize}
\end{fact}

On the one hand, conditions (a) and (b) can lead to an $(n, 2w-2, w)_3$ code trivially. On the other hand, if any one of them are violated, there must be two different codewords in $\mathcal C$ such that the $\ell_1$-distance between them is less than $2w-2.$ We will use the language in graph packings to rewrite this fact in the next subsection.

\subsection{From Codes to Graph Packings}

A {\it graph} $G$ is a pair $(V,E)$, where $V$ is a finite set of {\it vertices} and $E$ is a set of $2$-subsets of $V$, called {\it edges}. The {\it order} of $G$ is the number of vertices, denoted by $|G|$. Correspondingly, $e(G)$ means the numbers of edges. The {\it degree} of a vertex $v$ in $G$ is the number of edges containing $v$, and it is denoted by $d_G(v).$ A graph is called {\it complete} if each pair of vertices is contained in an edge. We write the complete graph with order $n$ as $K_n$.

We have given a one-to-one correspondence between a ternary vector and its labeled support. Here we give another one-to-one correspondence between a ternary vector and a colored complete subgraph. Given a vector $\bm u\in I_3^n$, we map it to a complete subgraph (clique) $K_{\bm u}$ of $K_n$ with vertex set $supp(\bm u)$ equipped with a $2$-coloring, such that the vertex $i\in supp(\bm u)$ is colored by $\bm u_i$.  By these correspondences, a vector in $I_3^n$, a labeled support in $[n]\times[2]$, and a $2$-colored clique are indeed the same objects, so they share the same notations like weight and type.

Next we will restate Fact \ref{Transfer} in the graph packing language. The {\it packing} of a graph $G$ is a set of its subgraphs $\M=\{G_1, G_2, \ldots, G_s\}$ such that  any edge of $G$ appears in at most one subgraph in $\M$. If the packing $\M$ covers all edges in $G$, we call it a {\it decomposition} of $G$. If all $G_i, i\in[s]$ are isomorphic to some graph $H$, then we call $\M$ an {\it $H$-packing.} Further if an $H$-packing $\M$ is also a decomposition, we call it an {\it $H$-decomposition.}
We have the following corollary from Fact \ref{Transfer}.
\begin{corollary}\label{transfer}
A ternary CWC code $\mathcal C$ with weight $w$ is an $(n, 2w-2, w)_3$ code if and only if the following conditions are satisfied:
\begin{itemize}
\item[$(a')$] The set $\{K_{\bm u}:\bm u\in \mathcal C\}$  forms a packing of $K_n$.
\item[$(b')$] Each vertex $i\in [n]$ is colored by $2$ in $\mathcal C$ at most once, i.e. there exists at most one codeword $\bm u\in \mathcal C$ such that $i$ is colored by 2 in $K_{\bm u}$.
\end{itemize}
\end{corollary}
The proof of Corollary \ref{transfer} is straightforward from Fact \ref{Transfer}. By Corollary \ref{transfer}, the problem of finding an optimal $(n, 2w-2, w)_3$ code turns to the one of finding a maximum-sized packing of $K_n$ by colored cliques with the same weight $w$ satisfying condition $(b')$. Here we point out that the number of codewords of type $1^a2^b$ with $b>0$ in an $(n, 2w-2, w)_3$ code  is no more than $n$ due to condition $(b')$. This means the set of codewords of type $1^w$ is an absolute majority when the upper bound $B(n)+n$ is attained for large $n$.

Our desired CWC is constructed by a packing of $K_n$ with two parts $\mathcal P\cup \mathcal S$, where $\mathcal P$ consists of all cliques of order $w$ in the packing, and $\mathcal S$ consists of all others. So $\mathcal P$ corresponds to all codewords of type $1^w$ and $\mathcal S$ corresponds to at most $n$ codewords of type $1^a2^b$ with some $b>0$. The theorem below given in \cite{alon1998packing} guarantees the existence of $\mathcal P$ as long as we have a suitable $\mathcal S$, which has  been used in \cite{chen2020optimal} for a construction of optimal CWCs of distance $2w-2$ for several congruence classes. We need some notations to state this theorem.

For a graph $H$ without isolated vertices, let $\gcd(H)$ denote the greatest common divisor of the degrees of all vertices of $H$. A graph $G$ is called {\it $d$-divisible} if  $ \gcd(G)$ is divisible by $d$, while $G$ is called {\it nowhere $d$-divisible} if no vertex of $G$ has degree divisible by  $d$. The $H$-packing number of $G$, denoted by $P(H, G)$ is the maximum cardinality of an $H$-packing of $G$.

\begin{theorem}\label{alon}\cite{alon1998packing}
  Let $H$ be a graph with $h$ edges, and let $\gcd(H)=e$. Then there exist $N=N(H)$, and $\varepsilon=\varepsilon(H)$ such that for any $e$-divisible or nowhere $e$-divisible graph $G=(V,E)$ with $n>N(H)$ vertices and $\delta(G)>(1-\varepsilon(H))n$, \[P(H,G)=\left\lfloor\frac{\sum_{v\in V}\alpha_v}{2h}\right\rfloor,\] unless when $G$ is $e$-divisible and $0<|E|\pmod{h}\leq \frac{e^2}{2}$, in which case  \[P(H,G)=\left\lfloor\frac{\sum_{v\in V}\alpha_v}{2h}\right\rfloor-1.\]  Here, $\alpha_v$ is the degree of vertex $v$, rounded down to the closest multiple of $e$.
\end{theorem}

\begin{remark}\label{alonremark}
If we choose $H=K_{w}$, then $gcd(H)=w-1$. For any given positive integer $n$, if we can find a subgraph $G_n$ of $K_n,$ such that $e(H)|e(G_n)$ and for any vertex $v\in K_n$,
$$\left\{
    \begin{array}{lr}
    {d_{G_n}(v)\ge n-1-2w^2;}&\\
    {(w-1)|d_{G_n}(v),}&
    \end{array}
\right.
$$
then there exists an integer $N'=N'(H)$ such that if $n>N'$, an $K_w$-decomposition of $G_n$ is guaranteed.

In fact, by the condition $d_{G_n}(v)\ge n-1-2w^2$, we have
$\delta(G_n)\ge n-1-2w^2 > (1-\varepsilon)n$ for any $n>\frac{1+2w^2}{\varepsilon}$.
This remark can be immediately derived from Theorem \ref{alon}.
\end{remark}

By Corollary~\ref{transfer} and Remark~\ref{alonremark}, to construct an $(n, 2w-2, w)_3$ code of size $B(n)+n$, it is sufficient to find a subgraph $S$ of $K_n$ such that the following conditions are satisfied:
\begin{itemize}
\item[(1)] $G_n=K_n-S$ satisfies the conditions in Remark \ref{alonremark}, so that $G_n$ has a $K_w$-decomposition $\P$. The cliques in $\P$ will produce all codewords of type $1^w$.
    \item[(2)] $S$ has a packing $\S$ of size $B(n)+n-|\P|$ containing cliques of order less than $w$, such that each clique could be colored by $\{1,2\}$  independently to become a colored clique of weight $w$, and each vertex $i\in [n]$ is colored by $2$ at most once among all cliques in $\S$. The cliques in  $\S$ will yield all codewords of type $1^a2^b$ for some $b>0$.\end{itemize}
         If $\S$ is also a decomposition of $S$, we say that the desired code is a \emph{balanced code}, that is, the set $\P\cup \S$ of cliques corresponding to all codewords  forms a decomposition of $K_n$. How to construct such a graph $S$ with a packing $\mathcal S$ by colored cliques, or equivalently the sub-code $\S\subset \mathcal C$ consisting of only codewords of type $1^a2^b$ for some $b>0$,  needs some techniques. The most important tool used throughout this paper is the modular Golomb ruler.

%From now on, we sometimes may replace $[n]$, the vertex set of the whole graph $K_n$, with other sets of cardinality $n$ to show a clearer view of the structure we construct. For example we may pick some vertices $v_1, v_2, \ldots, v_m$ out for some integer $m<n$ and let the rest $n-m$ vertices be the elements in ring $\bbZ_{n-m}.$ In our work, we use integers in $[0, n-1]$ to represent an element in $\bbZ_n$ for any $n\in\bbZ,$ and then a nature order of elements is derived, i.e., we call $a<b$ if and only if $a<b$ as integers in $[0, n-1].$ The ring of form $\bbZ_n$ is widely used in our constructions since we can define a modular Golomb ruler on it.
\subsection{Modular Golomb rulers}
An $(n,w)$ \emph{modular Golomb ruler} \cite{Shearer2006difference} is a set of $w$ integers $\{a_1 , a_2 ,\ldots, a_w\}\subset \bbZ_n$, such that all of
the differences, $\{a_i- a_j\mid 1 \leq i \neq j \leq w\}$, are distinct and nonzero. Notice that all calculations here are in $\bbZ_n$. The set $\{a_i- a_j\mid 1 \leq i \neq j \leq w\}$ is called the \emph{set of differences} of the modular Golomb ruler. Given an $(n,w-1)$ modular Golomb ruler $\{a_1, a_2,\ldots, a_{w-1}\}$, we construct $n$ codewords $\bm a_i=\{(a_1+i)_2, (a_2+i)_1,\ldots, (a_{w-1}+i)_1\}\subset \bbZ_n\times[2]$ of type $1^{w-2}2^1$ for any $i\in\bbZ_n$. Here the set of positions $[n]$ is replaced by $\bbZ_n$ with a natural order. These $n$ codewords have pairwise distances at least $2w-2$ by definition, and each position in $\bbZ_n$ is labeled by $2$ exactly once. Using  the language of graphs, these colored cliques of type $1^{w-2}2^1$ are edge-disjoint on $\bbZ_n$ and each vertex in $\bbZ_n$ is colored by $2$ exactly once among these colored cliques.

\begin{example}\label{exampleforgolomb}
The set $\{0, 1, 3\}$ is an $(8, 3)$ modular Golomb ruler. The eight codewords are $\bm a_0=\{0_2, 1_1, 3_1\},$ $\bm a_1=\{1_2, 2_1, 4_1\},$ $\bm a_2=\{2_2, 3_1, 5_1\},$ $\bm a_3=\{3_2, 4_1, 6_1\},$ $\bm a_4=\{4_2, 5_1, 7_1\},$ $\bm a_5=\{5_2, 6_1, 0_1\},$ $\bm a_6=\{6_2, 7_1, 1_1\}$ and $\bm a_7=\{7_2, 0_1, 2_1\}.$ Each vertex in $\bbZ_8$ is colored by $2$ exactly once.
\end{example}

We show an example to explain how we use modular Golomb ruler to get subgraphs $S$ and $G_n$ of $K_n$ and then lead to an optimal code. The example below is from \cite[Theorem VI.2.]{chen2020optimal}.

\begin{example}\label{egchen}
When $n=\Omega(w^2)$, there exists an $(n, w-1)$ modular Golomb ruler \cite{Shearer2006difference}. Let the whole graph $K_n$ be of vertex set $\bbZ_n.$ By our discussion above, the $n$ codewords $\bm a_i$, $i\in\bbZ_n$ from the modular Golomb ruler form a packing by $n$ colored cliques of type $1^{w-2}2^1$. Let $S=\cup_{i\in {\bbZ_n}}K_{\bm a_i}$, which has a trivial decomposition of size $n$, and let $G_n=K_n-S.$  For any $v\in \bbZ_n$, $d_{G_n}(v)=n-1-(w-1)(w-2).$ Further, there are $\frac{n(n-1)}2-n\frac{(w-1)(w-2)}2$ edges in $G_n$. When $n\equiv w \mod w(w-1),$ it is easy to check that $G_n$ satisfies all conditions in Remark \ref{alonremark} and then there exists a decomposition $\mathcal P$ of $G_n$ by colored cliques of type $1^w$ as long as $n>N'(w).$ With simple calculations we get $|\mathcal P|=B(n).$ Hence the code $\mathcal C=\{\bm a_i: i\in \bbZ_n\}\cup\mathcal P$ forms a balanced $(n, 2w-2, w)_3$ code of size $B(n)+n$ by Corollary \ref{transfer}. This lead to the following conclusion:

Given any $w\ge  5$, if $n\equiv w \mod w(w-1),$ $A_3(n, 2w-2, w)=B(n)+n$ for any sufficiently large integer $n.$
\end{example}

A finite set $B=\{b_1, b_2, \ldots, b_m\}\subset [n-1]$ is  \emph{free} from an $(n, w)$ modular Golomb ruler $\{a_1, \ldots, a_w\}$ if all $b_i, i\in[m]$, as elements of $\bbZ_n$, do not appear in the set of differences. Such an $(n, w)$ modular Golomb ruler is called \emph{$B$-free}. It is trivial to observe that the $(8, 3)$ modular Golomb ruler in Example \ref{exampleforgolomb} is $\{4\}$-free. Specially, by definition if $\{a_1, \ldots, a_w\}$ is a $B$-free $(n, w)$ modular Golomb ruler, so is the set $\{a_1+i, \ldots, a_w+i\}$  for all $i\in \bbZ_n.$

 %with a banning set $B$ is called a restricted $(n, w)$ modular Golomb ruler.
\begin{lemma}\label{golomb}
For any integer $w\geq 5$ and a  set $B\subset[n-1]$ with the largest element $b$, there exists a $B$-free $(n, w-1)$ modular Golomb ruler when $n\ge w(2b+w^2)$. Specially when $B$ is empty, we  define $b=0$.
\end{lemma}
\begin{proof} We will greedily find $w-1$ pairwise distinct integers $a_1, a_2, \cdots, a_{w-1}$ in $[0, n-1]$, which form a $B$-free $(n, w-1)$ modular Golomb ruler when viewed as elements of $\bbZ_n$.
%We show the existence of $a_k$ in the $k$th step, $k\in[w-1],$ such that $a_k>a_{k-1}$ and $\{a_1, \ldots, a_k\}$ forms a $B$-free $(n, k)$  modular Golomb ruler.
Start with $a_1=0$. The set $\{0\}$ is clearly a $B$-free $(n, 1)$ modular Golomb ruler.

For any $k\in[2, w-1]$, assume that  there exists pairwise distinct integers $a_1, \ldots, a_{k-1}$ in $[0, n-1]$ such that $\{a_1, \ldots, a_{k-1}\}$ forms a $B$-free  $(n, k-1)$ modular Golomb ruler. We consider the $k$th step.
If there exists an integer $v\in [0, n-1]$, such that both the following conditions are satisfied,
\begin{itemize}
\item for any $i\in[k-1],$ $b<v-a_i<n-b$;
\item for any $i\in[k-1],$ $v-a_i$ is not in the set of differences of $\{a_1, \ldots, a_{k-1}\}$,
\end{itemize}
then $\{a_1, \ldots, a_{k}\}$ is a $B$-free $(n, k)$ modular Golomb ruler after we set $a_k=v.$
There are at most $(k-1)(2b+1)$ integers in $[n-1]$ violate the first condition and at most $2(k-1)\binom{k-1}{2}$ integers violate the second one. When $n\ge w(2b+w^2)$, by pigeonhole principle such $v$ always exists.
\end{proof}

Before ending this section, we mention that in our constructions in the sections to follow, we usually use $\bbZ_{n-k} \cup \{b_1, \ldots, b_k\}$ with a natural order, to denote the set of positions of the code, or equivalently the vertex set of $K_n$ to be packed. This setting will help us to give a clearer prospect for the structure of our desired subgraph $S$ and its packing $\mathcal S$ with colored cliques.

\section{General Idea}\label{mainconsturctionfor01}

In this section, we illustrate the idea of our main construction of $\mathcal S$ by simple cases such that the resulting code is optimal and balanced. To make our constructions simple, we assume that  $\mathcal S$  only contains codewords of types $1^{w-2}2^1$ and $1^{w-4}2^2$. Let $x$, $y$ and $z$ denote the number of codewords of types $1^w$, $1^{w-2}2^1$ and $1^{w-4}2^2$, respectively.

Based on the above assumption, we first figure out the valid numbers $x$, $y$, and $z$ of codewords with different types by the balanced property. We start with $x=B(n)$, $y=n$, and $z=0$,  which in total reaches the upper bound. Note that the current $S$ is the union of cliques of order $w-1$ from these $n$ codewords of type $1^{w-2}2^1$. Let $\ell$ be the number in the range $0\leq \ell< {w\choose 2}$ such that \begin{equation}\label{1l}\ell\equiv {n \choose 2}-n{w-1 \choose 2}  \mod {{w\choose 2}},\end{equation} that is
\begin{equation}\label{2l}2\ell\equiv n(n-1)-n(w-1)(w-2) \mod {w(w-1)}.\end{equation}
 Let $G_n=K_n-S$, then the right hand side  of Equation~(\ref{1l}) is the number of edges in $G_n$. If $\ell\neq 0$, i.e., $e(K_w)\nmid e(G_n)$, then there does not exist a $K_w$-decomposition of $G_n$, that is the code can not be balanced. To make it balanced, we need to modify the initial numbers $x$, $y$ and $z$ so that all cliques consume $\ell$ extra edges. We do the following two operations without changing the sum $x+y+z$.
\begin{itemize}
\item[~~(O1)] Replace one codeword of type $1^{w-2}2^1$ by a codeword of type $1^w$ to consume $(w-1)$ more edges, i.e., $x\rightarrow x+1$, $y\rightarrow y-1$;
\item[  (O2)] Replace two codewords of type $1^{w-2}2^1$ by a codeword of type $1^w$ and one of type  $1^{w-4}2^2$ to consume  one more edge, i.e., $y\rightarrow y-2$, $x\rightarrow x+1$ and $z\rightarrow z+1$.
\end{itemize}
 By these two operations, we are able to figure out some values of $x$, $y$ and $z$, such that the total number of codewords $x+y+z=B(n)+n$ still holds, and  $\ell$ extra edges are all consumed,  that is, they are  valid for a balanced code. Suppose that we do $a$ times of (O1) and $b$ times of (O2), then $(a,b)$ is just any nonnegative integral solution to \begin{equation}\label{ab}a(w-1)+b=\ell.\end{equation} Then $x=B(n)+a+b$, $y=n-a-2b$, $z=b$. We always choose $b$ as the smallest nonnegative integer such that Equation~(\ref{ab}) has a nonnegative integral solution.

Next, we try to construct $y$ codewords of type $1^{w-2}2^1$ and $z$ codewords of type $1^{w-4}2^2$ to form a new subgraph $S$. It is clear now that the new $G_n=K_n-S$ will satisfy $e(K_w)\mid e(G_n)$. To apply  Theorem~\ref{alon} or Remark~\ref{alonremark}, we need that $G_n$ is $(w-1)$-divisible, or equivalently, for all $v\in V(K_n)$, $n-1-d_S(v)\equiv 0 \mod {(w-1)}$. For any $v\in V(K_n),$ we define $y_v$ and $z_v$  the number of cliques of type $1^{w-2}2^1$ and $1^{w-4}2^2$, respectively, containing $v$ in $\mathcal S$. Thus $d_S(v)=y_v(w-2)+z_v(w-3)$ and $$n-1-d_S(v)=n-1-y_v(w-2)-z_v(w-3)\equiv n-1+y_v+2z_v\equiv 0\mod {(w-1)}.$$ Denote $R(v)\triangleq y_v+2z_v$.
For simplicity, we assume $R(v)$ can only be the least two non-negative integers satisfying this equation. Assume
that $n\equiv t\mod (w-1)$ with $0\leq t \leq w-2$. Let
%\begin{equation}\label{rm}
\begin{equation}\label{rm}
R_m\triangleq\left\{
\begin{array}{ll}
1, & \text{ if }t=0;\\
0, & \text{ if } t=1;\\
w-t, & \text{ if } t\geq 2.
\end{array}
\right. \end{equation}
Then $R(v)\in \{R_m,R_m+w-1\}$.
%R(v)\in &\{w-t, 2w-t-1\}\text{ if } t\geq 2, R(v)\in \{0, w-1\}\text{ if } t=1,\text{ and }R(v)\in \{1, w\}\text{ if }t=0.\\
%& \text{Let $R_m$ denote the minimum value of $R(v)$ for each case.}

%\end{equation}

When constructing $\mathcal S$, we follow an additional rule that each vertex  can be in at most one clique of type $1^{w-4}2^2,$ i.e., $z_v=0$ or $1$ for all $v.$ Since there are in total $z=b$ cliques of type $1^{w-4}2^2$, $|\{v\in V(K_n): z_v=1\}|=b(w-2).$  Note that $d_S(v)=(w-2)R(v)-(w-1)z_v$, then $\sum_{v}d_S(v) =\sum_{v}(w-2)R(v)-\sum_{\{v: z_v=1\}}(w-1) =\sum_v(w-2)R(v)-b(w-1)(w-2).$
Let $c$ be the number of vertices having $R(v)=R_m$ in $S$. By the handshaking lemma, $\sum_{v}d_S(v)=2e(S)$,
 and we have
\begin{equation*}
c(w-2)R_m+(n-c)(w-2)(R_m+w-1)-b(w-2)(w-1)=(n-a-2b)(w-1)(w-2)+b(w-2)(w-3).
\end{equation*}
That is \begin{equation}\label{dsv}
c(w-1)=nR_m+a(w-1)+2b.
\end{equation}
So far, we have two equations (\ref{ab}) and~(\ref{dsv}), and three undetermined parameters $a,b,c$.
It is possible for us to find a common nonnegative integral solution $(a,b,c)$ to both equations. Note that the value $c$ reflects the distribution of vertices in codewords of type $1^{w-2}2^1$ and $1^{w-4}2^2$ in some way. In Example~\ref{egchen}, an $(n, w-1)$ modular Golomb ruler produces $n$ codewords of type $1^{w-2}2^1$ such that each vertex occurs in exactly $(w-1)$ such codewords, which contributes $(w-1)$ to $y_v$ or $R(v)$ for each $v\in \bbZ_n$.

Based on the above valid choices of $a,b,c$, we have the following result.

\begin{lemma}\label{condofS}
Given an integer $w\ge 5$ and a sufficiently large $n\equiv t\mod (w-1)$ for some $t\in [0,w-2]$, let $a, b, c$ and $ R_m$ be non-negative integers satisfying Equations (\ref{ab})-(\ref{dsv}).  If $\S$ is an $(n,2w-2,w)_3$ code consisting of $n-a-2b$ codewords of type $1^{w-2}2^1$ and $b$ codewords of type $1^{w-4}2^2$, such that
\begin{itemize}
\item[(1)] there are exactly $c$ positions (or vertices $v$ of $K_n$) with $R(v)=R_m$, all others are $R_m+w-1$;
\item[(2)] the sets $supp(\bm u)$ for all codewords $\bm u$ of type $1^{w-4}2^2$ are disjoint,
\end{itemize}
then there is an optimal balanced $(n, 2w-2, w)_3$ code and $A_3(n, 2w-2, w)=B(n)+n$.
\end{lemma}

\begin{proof}
Let $S=\cup_{\bm a\in \S}K_{\bm a}$ be a graph and $G_n=K_n-S.$ We apply Remark \ref{alonremark} to prove that $G_n$ has a $K_w$-decomposition of size $B(n)+a+b$.

First, for any $v\in V(K_n),$ $d_{G_n}(v)=n-1-d_S(V)\ge n-1-R(v)(w-2)>n-1-2w^2.$ Specially, since $d_{G_n}(v)\equiv n-1+R(v) \mod (w-1),$ we have $d_{G_n}(v)\equiv 0\mod (w-1),$ i.e. $G_n$ is $(w-1)$-divisible.

Further, we have
$$
\begin{aligned}
\frac{e(G_n)}{e(K_w)} &=\frac{2e(G_n)}{2e(K_w)}\\
&=\frac{n(n-1)-n(w-1)(w-2)}{w(w-1)}+\frac{(a+2b)(w-1)(w-2)-b(w-2)(w-3)}{w(w-1)}\\
&=B(n)+\frac{2\ell+(a+2b)(w-1)(w-2)-b(w-2)(w-3)}{w(w-1)}\\
&=B(n)+a+b,
\end{aligned}
$$
where the last step is by (\ref{ab}).

With all conditions satisfied, Remark \ref{alonremark} gives us an integer $N=N(w)$ such that if $n>N,$ there exists a $K_w$-decomposition $\mathcal P$ of $G_n$. By vertex-coloring all cliques in $\mathcal P$ with color $1$, $\mathcal P$ is a code with $B(n)+a+b$ codewords of type $1^w.$

Finally, by Corollary \ref{transfer} and Lemma \ref{upperbound}, $A_3(n, 2w-2, w)=B(n)+n$ and the code $\mathcal C=\S\cup\mathcal P$ is an optimal balanced $(n, 2w-2, w)_3$ code.
\end{proof}

%to show the existence of codewords of type $1^w$ from the $K_w$-decomposition of the new $G_n$.
By Lemma~\ref{condofS}, to show that an optimal balanced $(n, 2w-2, w)_3$ code exists with $A_3(n, 2w-2, w)=B(n)+n$, we only need to find an $(n, 2w-2, w)_3$ code $\S$ satisfying Lemma~\ref{condofS}.
Next, we take $n\equiv 0,1\mod (w-1)$ to show the way of computing the valid triple $(a,b,c)$, and then the numbers $x, y$ and $z$.  By computing the right hand side of Equation~(\ref{2l}), we know that $(w-1)\mid 2\ell$. Assume that $(w-1)\mid \ell$ in the following constructions, since in this case we only need to do operation (O1). We will give constructions for code $\S$ satisfying Lemma~\ref{condofS}.

%In this section, we present two similar constructions of $(n, 2w-2, w)_3$ code which can achieve the theoretical bound for sufficiently large $n$ satisfying $n\equiv 0,1 \mod (w-1)$ and {\it property $(A)$}: $n(n-1)-n(w-1)(w-2)\equiv 2k(w-1) \mod w(w-1)$ for some $k\in[0, \lceil\frac w2\rceil-1]$. In our constructions, only codewords of type $1^w$ and $1^{w-2}2^1$ are considered. By the discussion on Corollary \ref{alonremark}, we point out all codewords of type $1^{w-2}2^1$ in $\mathcal S$ such that no vertex can be colored by 2 more than once. Then we delete all those cliques in $\mathcal S$ from $K_n$ and show the remaining graph $G_n$ has a $K_w$-decomposition.

\subsection{$n\equiv 1\mod (w-1)$ and $(w-1)\mid \ell$}\label{seciii:a}
Let $n=h(w-1)+1$, and $2\ell=2k(w-1)$  for some $k\in[0, \lceil\frac w2\rceil-1]$.
Then we do operation (O1) $k$ times to consume $\ell=k(w-1)$ edges, that is $a=k$ and $b=0$. The value $R(v)\in \{0, w-1\}$ and $R_m=0$ in this case, so $c=a=k$ by (\ref{dsv}). The candidates of numbers of codewords of types  $1^w$, $1^{w-2}2^1$ and $1^{w-4}2^2$ are $x=B(n)+k$, $y=n-k$ and $z=0$.  Let $K_n$ be the complete graph with vertex set $\bbZ_{n-k}\cup \{b_i:i\in[k]\}$. The code $\S$ is constructed as follows.

% We only consider the case when $n$ and $w$ satisfy property $(A)$. Since $n=(n-k)+k$, all $n$ vertices of $K_n$ can be denoted as the union of $\bbZ_{n-k}$ and $k$ other vertices $\{b_i\}_{i\in[k]}$.

When $h$ is large enough such that $n-k>w^3$, by Lemma \ref{golomb}, there exists an $(n-k, w-1)$ modular Golomb ruler $\{a_1, a_2, \ldots, a_{w-1}\}$ on $\bbZ_{n-k}$.
Define codewords $\bm a_i=\{(a_1+i)_2, (a_2+i)_1, (a_3+i)_1, \ldots, (a_{w-1}+i)_1\}$ of type $1^{w-2}2^1$ for $i\in \bbZ_{n-k}$. This gives us $n-k$ cliques of type $1^{w-2}2^1$ which form a packing of $K_n$.
Define $\mathcal S=\{{\bm a_i}: i\in \bbZ_{n-k}\}$. Each vertex $v$ in $\bbZ_{n-k}$ is colored by $2$ exactly once, and contained in $w-1$ different colored cliques of $\mathcal S$, that is $R(v)=w-1$. The $c=k$ vertices $b_i$ are not contained in any colored clique of $\mathcal S$, that is $R(v)=0$.
So $\S$ is an $(n,2w-2,w)_3$ code satisfying Lemma~\ref{condofS}, thus $A_3(n, 2w-2, w)=B(n)+n$ for sufficiently large $n$.

% , and let graph $S=\cup_{i\in{\bbZ_{n-k}}}K_{\bm a_i}$ It is left to check that $G_n=K_n-S$  satisfies all conditions in Remark~\ref{alonremark}. Actually, $d_{G_n}(v)=n-1$ for  $v\not\in \bbZ_{n-k}$ and $d_{G_n}(v)=n-1-(w-1)(w-2)$ for  $v\in \bbZ_{n-k}$. This leads to $d_{G_n}(v)\ge n-1-2w^2$ for any vertex $v$ and $G_n$ is $(w-1)$-divisible. The fact that  $e(K_{w})|e(G_n)$ is clear from our operations. Then by Remark~\ref{alonremark}, there is a $K_w$-decomposition $\P$ of $G_n$ with size
%$$
%\begin{aligned}
%\frac{e(G_n)}{e(K_w)} &=\frac{2e(G_n)}{2e(K_w)}\\
% &=\frac{n(n-1)-n(w-1)(w-2)}{w(w-1)}+\frac{k(w-1)(w-2)}{w(w-1)}\\
% &=\lfloor\frac{n(n-1)-n(w-1)(w-2)}{w(w-1)}\rfloor+\frac{2k(w-1)+k(w-1)(w-2)}{w(w-1)}\\
% &=B(n)+k
%\end{aligned}$$ when $h$ is sufficiently large.
%By vertex-coloring all cliques in $\P$ by $1$, we have produced a set of colored cliques $\mathcal P$, that is,  $B(n)+k$ codewords of type $1^w$ as desired.

%Finally, $\mathcal C=\mathcal P\cup \mathcal S$ forms a decomposition of $K_n$ by colored cliques with size $B(n)+n$, satisfying condition $(b')$ in Corollary \ref{transfer}. Applying Corollary \ref{transfer} and Lemma \ref{upperbound}, $\mathcal C$ is an optimal balanced $(n, 2w-2, w)_3$ code and .

\subsection{$n\equiv 0\mod (w-1)$ and $(w-1)\mid \ell$}\label{seciii:b}
 Let $n=h(w-1)$, and $2\ell=2k(w-1)$  for some $k\in[0, \lceil\frac w2\rceil-1]$.
Then we do operation (O1) $k$ times to consume $\ell=k(w-1)$ edges, that is $a=k$ and $b=0$. The value $R(v)\in \{1, w\}$ and $R_m=1$ in this case, so $c=h+a=h+k$.  The numbers of codewords of types  $1^w$, $1^{w-2}2^1$ and $1^{w-4}2^2$ will be $x=B(n)+k$, $y=n-k$ and $z=0$.

Let $n'=n-h-k$, and let the vertex set of $K_n$ be the union of three parts:
$\bbZ_{n'}$, $\{b_i\}_{i\in[h]}$ and $\{c_j\}_{j\in[k]}.$ The set $\{b_i\}_{i\in[h]} \cup \{c_j\}_{j\in[k]}$ will be the $c=h+k$ vertices with $R(v)=1$. The code $\S$ is constructed as follows.

When $h$ is large enough such that $n'>2w^3$, by Lemma \ref{golomb} there exists a $[w-1]$-free $(n', w-1)$ modular Golomb ruler $\{a_1, a_2, \ldots, a_{w-1}\}$. From this, we get $n'$ codewords of type $1^{w-2}2^1$, $\bm a_i=\{(a_1+i)_2, (a_2+i)_1, (a_3+i)_1, \ldots, (a_{w-1}+i)_1\}$ for $i\in \bbZ_{n'}$. %Define $\mathcal S'=\{\bm s_i: i\in[0, n'-1]\}$.
The other $h$ codewords $\bm s_i,$ $i\in[h]$ of type $1^{w-2}2^1$ are constructed as follows, whose supports form a partition of the vertex set of $K_n$. When $i\in[k]$, $\bm s_i=\{(b_i)_2, (c_i)_1\}\cup[(w-3)(i-1), (w-3)i-1]_1$; when $i\in[k+1, h]$, $\bm s_i=\{(b_i)_2\}\cup[(w-2)(i-1)-k, (w-2)i-k-1]_1$. Here, $[\cdot,\cdot]_1$ means that all integers in the interval are labeled by $1$. %As for the coloring of $\bm s_i,$ only the vertex $b_i$ is colored by $2$  for each $i\in[h]$.

Let $\mathcal S=\{{\bm a_i}: i\in \bbZ_{n'}\}\cup\{{\bm s_i}: i\in[h]\}$. It is clear that each vertex of $K_n$ is colored by $2$ at most once. We claim that $\S$ as a set of colored cliques forms a packing of $K_n$. If it is false, there exists two different $\bm a, \bm b\in \S$, such that  $K_{\bm a}$ and $K_{\bm b}$ have one edge $uv$ in common. In other words, $u$ and $v$ are both in $supp(\bm a)\cap supp(\bm b).$ As $\{K_{\bm a_i}: i\in \bbZ_{n'}\}$ and $\{K_{\bm s_i}: i\in[h]\}$ are both packings already, we assume $\bm a=\bm s_i$ and $\bm b=\bm a_j$ for some $i\in[h]$ and $j\in \bbZ_{n'}.$ Thus, both $u$ and $v$ are in $\bbZ_{n'}$ and we may assume $u>v$ as integers. Since $u$ and $v$ are in $supp(\bm s_i)$, $u-v\in [w-1],$ which contradicts to the fact that $supp(\bm a_j)$ is a $[w-1]$-free modular Golomb ruler. So $\S$ is an $(n,2w-2,w)_3$ code. Further,
each vertex $v\in\bbZ_{n'}$ is contained in $w$ different colored cliques in $\mathcal S$, that is $R(v)=w$; each $v\not\in\bbZ_{n'}$ is contained in only one colored clique in $\mathcal S$, that is $R(v)=1$. So $\S$ satisfies Lemma~\ref{condofS}, and $A_3(n, 2w-2, w)=B(n)+n$ for sufficiently large $n$.
%From our construction of code $\mathcal S$ above, all vertices are colored by 2 at most once in it.
%Let $S=\cup_{\bm a\in \mathcal S}K_{\bm a}$ and $G_n=K_n-S.$ We check that graph $G_n$ satisfies all conditions in Remark \ref{alonremark}. This leads to $d_{G_n}(v)\ge n-1-2w^2$ for all vertices and further $G_n$ is $(w-1)$-divisible.
%
%From our operations, we can still get $\frac{e(G_n)}{e(K_w)} = B(n) +k$ and the calculation process is much the same as in the case when $n\equiv 1\mod (w-1)$. It is routine to use Remark \ref{alonremark} to show there exists a decomposition $\mathcal P$ of $G_n$ by colored cliques of type $1^w$ with cardinality $B(n)+k$ when $n>N'$ for some $N'=N'(w).$ By Corollary \ref{transfer} and Lemma \ref{upperbound} the theoretical upper bound is tight, and the code $\mathcal C=\mathcal P\cup\mathcal S$ forms an optimal balanced $(n, 2w-2, w)_3$ code.

%\subsection{The Exceptional Cases}

\section{Completing the Cases $n\equiv 0,1\mod (w-1)$}\label{additionalconstructionA}

Subsections~\ref{seciii:a} and \ref{seciii:b} give us two similar constructions of the required $(n, 2w-2, w)_3$ codes $\S$ when $n\equiv 0,1\mod (w-1)$ and $(w-1)\mid \ell$. The property $(w-1)\mid \ell$ always holds when $w$ is even. In this section, we assume that $w$ is odd,  and consider constructions of optimal  $(n, 2w-2, w)_3$ codes when $n\equiv 0, 1\mod (w-1)$ and $(w-1)\nmid \ell$.
\subsection{$n\equiv 0\mod (w-1)$ and $(w-1)\nmid \ell$}
Let $n=h(w-1)$ and $2\ell=(2k+1)(w-1)$ for some $k\in[0, \frac{w-3}2].$ Then we do operation (O1) $k$ times and operation (O2) $\frac{w-1}2$ times to consume $\ell=k(w-1)+\frac{w-1}2$ edges, that is $a=k$ and $b=\frac{w-1}2.$ The value $R(v)\in\{1, w\}$ and $R_m=1$ in this case, and $c=\frac{n}{w-1}+a+\frac{2b}{w-1}=h+k+1.$ So the candidates of numbers of codewords of type $1^w$, $1^{w-2}2^1$ and $1^{w-4}2^2$ are $x=B(n)+k+\frac{w-1}2$, $y=n-k-(w-1)$ and $z=\frac{w-1}2.$

Let $K_n$ be the complete graph with vertex set $\bbZ_{n'} \cup B\cup C$, where $n'=n-\frac{(w-1)(w-2)}2-(h+k+1)$. The set $B$ has $\frac{(w-1)(w-2)}2$ vertices,  $\bar{b}_i$, $i\in[w-1]$ and $b_j, j\in[\frac{(w-1)(w-4)}2]$. The set $C$ consists of $h+k+1$ vertices,  $\bar{c}_i$, $i\in[h+1]$ and $c_j$, $j\in[k]$. An $(n, 2w-2, w)_3$ code $\S$ consisting of $n-k-(w-1)$ codewords of type $1^{w-2}2^1$ and $\frac{w-1}2$ codewords of type $1^{w-4}2^2$ is given below.
%In the next construction, we will construct in total and These colored cliques will serve as the elements in $\mathcal S$.

\begin{construction}\label{6differentkinds}
Suppose $h$ is large enough so that $h\ge k+w^3$. We construct six disjoint classes of codewords below. %Only codewords in (1) are of type $1^{w-4}2^2.$
%For convenience, we define $n_0=\frac{(w-1)(w-2)(w-4)}2$, $\alpha=(w-2)(w-1)+(w-3)\frac{(w-4)(w-1)}2$, $\beta=n_0+(w-3)\alpha$ and $\gamma=\beta+(w-2)(h-k-\alpha+1).$ Note that $n_0< \beta< \gamma$, and $\gamma+k(w-3)=n'$ by simple computation.
\begin{itemize}
\item[(1)] The   $\frac{w-1}2$  codewords of type $1^{w-4}2^2$ are
    $$\bm z_i=\{(\bar{b}_{2i-1})_2, (\bar{b}_{2i})_2\}\cup\{(b_{(w-4)(i-1)+1})_1, (b_{(w-4)(i-1)+2})_1, \ldots, (b_{(w-4)i})_1\},\text{ for all } i\in\left[\frac{w-1}2\right].$$
    Let $\mathcal Z=\{\bm z_i: i\in[\frac{w-1}2]\}.$ Note that $supp(\bm z_i)$, $i\in[\frac{w-1}2]$ form a partition of the set $B$, and  each $\bar{b}_i$, $i\in[w-1]$ is labeled by $2$ exactly once.

\item[(2)] Since $n'>2w^3$, there exists a $[w-1]$-free $(n', w-1)$ Golomb ruler $\{a_1, a_2, \ldots, a_{w-1}\}$, which derives the first $n'$  codewords of type $1^{w-2}2^1$,
    $$\bm a_i=\{(a_1+i)_2, (a_2+i)_1, (a_3+i)_1, \ldots, (a_{w-1}+i)_1\}, \text{ for all } i\in \bbZ_{n'}.$$
    Let $\mathcal Y_1=\{\bm a_i: i\in\bbZ_{n'}\}$. Each vertex $v\in \bbZ_{n'}$ appears in exactly $w-1$ cliques in $\mathcal Y_1$ and is labeled by $2$ in exactly one codeword.

\item[(3)] Let $n_0=\frac{(w-1)(w-2)(w-4)}2$. Since $n'>n_0$, the next $\frac{(w-1)(w-4)}2$ codewords of type $1^{w-2}2^1$ are
    $$\bm s_i=\{(b_i)_2\}\cup[(w-2)(i-1),  (w-2)i-1]_1, \text{ for all } i\in \left[\frac{(w-1)(w-4)}2\right].$$
    Let $\mathcal Y_2=\{\bm s_i: i\in [\frac{(w-1)(w-4)}2]\}$. Note that $supp(\bm s_i)$, $i\in [\frac{(w-1)(w-4)}2]$ are pairwise disjoint, covering all $b_i$ in $B$, and elements in $[0,n_0-1]\subset \bbZ_{n'}$. Further, each $b_i, i\in[\frac{(w-1)(w-4)}2]$ is labeled by $2$ exactly once.

\item[(4)] Let $\alpha=(w-2)(w-1)+(w-3)\frac{(w-4)(w-1)}2$. Note that $n'\ge n_0+\alpha(w-3)$ and $h\ge \alpha+k$. We consider a sequence $r_1, r_2, \ldots, r_\alpha$ with elements in $B$ as follows. When $i\le (w-2)(w-1)$, $r_i=\bar{b}_{\lceil\frac i{w-2}\rceil}$; when $(w-2)(w-1)<i\le \alpha$, $r_i=b_{\lceil\frac {i-(w-1)(w-2)}{w-3}\rceil}.$ In this sequence, each $\bar{b}_i$, $i\in [w-1]$ appears $w-2$ times and each $b_i$, $i\in[\frac{(w-4)(w-1)}{2}]$ appears $w-3$ times.
    The third class of codewords of type  $1^{w-2}2^1$ are
    $$\bm y_i=\{(\bar{c}_i)_2, (r_i)_1\}\cup[n_0+(i-1)(w-3),  n_0+i(w-3)-1]_1,\text{ for all } i\in[\alpha].$$
    Let $\mathcal Y_3=\{\bm y_i: i\in[\alpha]\}.$ Each element from  $[n_0,\beta-1]\subset \bbZ_{n'}$ appears in exactly one clique in $\mathcal Y_3$, where $\beta=n_0+(w-3)\alpha$. Each $\bar{c}_i$, $i\in[\alpha]$ is labeled by $2$ once.

\item[(5)] Since $h\ge \alpha+k$, $h-k-\alpha+1$ is a positive integer. The fourth class of codewords of type  $1^{w-2}2^1$ are
    $$\bm y_{\alpha+t}=\{(\bar{c}_{\alpha+t})_2\}\cup[\beta+(w-2)(t-1),  \beta+(w-2)t-1]_1, \text{ for all } t\in[h-k-\alpha+1].$$
    Let $\mathcal Y_4=\{\bm y_i: i\in[\alpha+1, h-k+1]\}.$ Each element from  $[\beta,\gamma-1]$ in $\bbZ_{n'}$ appears in exactly one clique in $\mathcal Y_4$, where $\gamma=\beta+(w-2)(h-k-\alpha+1)$. Each $\bar{c}_i$, $i\in[\alpha+1, h-k+1]$ is labeled by $2$ once.

\item[(6)]The last $k$ codewords of type  $1^{w-2}2^1$ are
    $$\bm y_{h-k+1+i}=\{(\bar{c}_{h-k+1+i})_2, (c_i)_1\}\cup[\gamma+(w-3)(i-1), \gamma+(w-3)i-1]_1,$$
   for $i\in [k]$.  Let $\mathcal Y_5=\{\bm y_{h-k+1+i}: i\in[k]\}.$ The sets $supp(\bm y_{h-k+1+i}), i\in[k]$ are pairwise disjoint, and cover all the remaining elements in $C$ and $\bbZ_{n'}$. Further, each $\bar{c}_i$, $i\in[h-k+2, h+1]$ is labeled by $2$ once.

    %Each vertex in $[\gamma, \bar n-1]$ appears in exactly one clique in $\mathcal Y_5$.

\end{itemize}
\end{construction}

%The proof for the claim is listed below.
%$$
%\begin{aligned}\gamma+k(w-3) &=\beta+(w-2)(y-k-\alpha+1)+k(w-3)\\
% &=n_0+(y+1)(w-2)-k-\alpha\\
% &=\frac{(w-1)(w-2)(w-4)}{2}-((w-2)(w-1)+\frac{(w-1)(w-3)(w-4)}{2})+(y+1)(w-2)-k\\
% &=n-y-k-\frac{(w-1)(w-2)}{2}-1\\
% &=\bar n.
%\end{aligned}$$

%It is simple to show that all codewords given in Construction \ref{6differentkinds} form a packing of $K_n.$ We leave the proof to readers.
 Finally, Let $\mathcal Y=\cup_{i\in[5]}\mathcal Y_i$ and $\mathcal S=\mathcal Z\cup\mathcal Y.$ All vertices in $K_n$ other than $c_i$, $i\in [k]$ are colored by $2$ exactly once in $\mathcal S.$
For any vertex $v\in C$, $R(v)=1.$ More specifically, $\bar{c}_i$ is only contained in  $\bm y_i$, while $c_j$ is only contained in $\bm y_{h-k+1+j}$.  Similarly, any vertex $v$ in $B$ is contained in exactly one clique in $\mathcal Z$ and $w-2$ cliques in $\mathcal Y,$ so $R(v)=w$. Each vertex $v\in \bbZ_{n'}$ is  contained in exactly $w$ cliques in $\mathcal Y$ and no clique in $\mathcal Z$, so $R(v)=w$.  It is left to check that $\S$ as a set of cliques forms a packing of $K_n$.

If it is false, there exists two different $\bm a, \bm b\in\mathcal S$ such that $K_{\bm a}$ and $K_{\bm b}$ have one edge $uv$ in common. Note that from modular Golomb ruler, $\Z\cup\mathcal Y_1$ is already a packing, and $\mathcal Y\backslash\mathcal Y_1$ is also a packing by simple verification. So we assume $\bm a \in \Z\cup\mathcal Y_1$ and $\bm b \in \mathcal Y\backslash\mathcal Y_1.$ If $\bm a\in\Z,$ all vertices in $\bm a$ are in  $B$ while no codeword in $\mathcal Y\backslash\mathcal Y_1$ has more than one element in $B$, hence $K_{\bm a}$ and $K_{\bm b}$ are edge disjoint. If $\bm a\in\mathcal Y_1$, then both $u$ and $v$ are in $\mathbb Z_{n'}.$ Suppose that $u>v$ as integers. Since $\{u, v\}\subset supp(\bm b)$ and $\bm b \in \mathcal Y\backslash\mathcal Y_1,$ by our construction $u-v\in[w-1]$, which contradicts to the fact that $supp(\bm a)$ is a $[w-1]$-free modular Golomb ruler.
%
%In total, $G_n$ is $(w-1)$-divisible, and $d_{G_n}(v)\ge n-1-2w^2$ for all vertex $v$. The final condition that $e(K_w)|e(G_n)$ is also correct.
%$$
%\begin{aligned}\frac{e(G_n)}{e(K_w)} &=\frac{n(n-1)-(w-1) e(K_{w-2})-2(n-k-w+1)e(K_{w-1})}{2e(K_w)}\\
% &=B(n)+\frac{(2k+1)(w-1)+2(k+w-1)e(K_{w-1})-(w-1) e(K_{w-2})}{2e(K_w)}\\
% &=B(n)+k+\frac{w-1}2.
%\end{aligned}$$
%
%By Remark \ref{alonremark}, there exists an integer $N=N(w)$ such that a $K_w$-decomposition $\P$ of $G_n$ with size $B(n)+k+\frac{w-1}2$ exists when $h>N$. We get a code $\mathcal P$ from vertex-coloring all cliques in $\P$ by $1$. Similarly, $\mathcal P\cup \mathcal S$ forms a balanced $(n, 2w-2, w)_3$ code of size $B(n)+n$ by Corollary \ref{transfer} and reaches the upper bound. Thus,

%By $A_3(n, 2w-2, w)=B(n)+n$ when $n$ is large enough.

 Combing the result above with Lemma~\ref{condofS} and Subsection~\ref{seciii:b}, we have the following theorem.

\begin{theorem}\label{theoremformod0case}
For any $w\ge 5$, there exists an integer $N=N(w)$ such that, for any integer $n>N$ and $n\equiv 0 \mod (w-1),$ $A_3(n, 2w-2, w)=B(n)+n$. Furthermore, there exists a balanced optimal $(n, 2w-2, w)_3$ code.
\end{theorem}

%Note that all three constructions we have given are balanced codes, and hence there is a trivial improvement for Theorem \ref{theoremformod0case}.
%
%\begin{corollary}\label{corollaryformod0case}
%For any $w\ge 5$, there exists an integer $N=N(w)$ such that, for any integer $n>N$ and $n\equiv 0 \mod (w-1),$ $A_3(n, 2w-2, w)=B(n)+n$ and a balanced optimal $(n, 2w-2, w)_3$ code exists.
%\end{corollary}

\subsection{$n\equiv 1 \mod (w-1)$ and $(w-1)\nmid \ell$}\label{additionalconstructionB}

A balanced optimal code may not exist for certain parameters $w$ and $n$. In this subsection, we will show the nonexistence of balanced optimal $(n, 2w-2, w)_3$ codes when $n\equiv 1 \mod (w-1)$ and $(w-1)\nmid \ell$. In other words, if there is an $(n, 2w-2, w)_3$ code of size  $B(n)+n$, then there must be at least one edge which is not covered by any colored clique. We first state that a code reaching the theoretical upper bound cannot be balanced.

\begin{theorem} \label{upperboundnonexisting}
Suppose that $n\equiv 1 \mod (w-1)$ and $n(n-1)-n(w-1)(w-2)\equiv (2k+1)(w-1) \mod w(w-1)$ for some $k\in[0, \frac{w-3}2]$. Then an $(n, 2w-2, w)_3$ code of size  $B(n)+n$ is not balanced.
\end{theorem}

%If we still use the same method as before hoping for a decomposition for a balanced code reaching the theoretical upper bound, results must be disappointing. We can prove that in this case such a code do not exist.

We need the following lemma to prove Theorem~\ref{upperboundnonexisting}.
\begin{lemma}\label{lemmaforupperboundnonexisting} Under the same parameters as in Theorem~\ref{upperboundnonexisting},
if there is a balanced $(n, 2w-2, w)_3$ code $\mathcal C$ of size $B(n)+n$, then $\mathcal C$ must contain a codeword which is neither of type $1^w$ nor of type $1^{w-2}2^1.$
\end{lemma}

\begin{proof}
Suppose to the contrary that $\mathcal C$ consists of only codewords of type $1^w$ and  $1^{w-2}2^1$, and $x$, $y$ are their numbers, respectively. Then
\begin{equation*}%\label{equation1}
x+y=B(n)+n.
\end{equation*}
Since $\mathcal C$ is balanced,
\begin{equation*}%\label{equation2}
xe(K_w)+ye(K_{w-1})=e(K_n).
\end{equation*}
Combining both equations, we get $y=\frac{wB(n)+n(w-\frac{n-1}{w-1})}2$. But substituting $B(n)=\frac{n(n-1)-n(w-1)(w-2)-(2k+1)(w-1)}{w(w-1)}$ into this expression,  we have $y=n-k-\frac 12$, which is not even an integer.
\end{proof}

\begin{proof}[Proof of Theorem~\ref{upperboundnonexisting}]Suppose that $\mathcal C$ is a balanced $(n, 2w-2, w)_3$ code of size $B(n)+n$. Let $y(i)$, $i\ge 0$ be the number of codewords in $\mathcal C$ with type $1^{w-2i}2^i$.
By Lemma \ref{lemmaforupperboundnonexisting},
there exists some $i\in[2, \lfloor\frac w2\rfloor]$ such that $y(i)\ne 0$. Consider the complete graph $K_n$ with vertex set $V$. Let $S$ be the union of all cliques from codewords of type $1^{w-2i}2^i$ with $i>0$, and let $s$ be the number of non-isolated vertices in $S$. Since a  codeword of type $1^{w-2i}2^i$ labels $i$ vertices by $2$,  the condition $(b')$  in Corollary \ref{transfer} gives us the following inequality.
\begin{equation}\label{equation3}
\sum_{i>0}i y(i)\le s.
\end{equation}

For any vertex $v\in V$, let $y_v(i)$ be the number of codewords of type $1^{w-2i}2^i$ whose support contains $v$. Since $\mathcal C$ is  balanced,  $G=K_n-S$ is $(w-1)$-divisible, and so is $S$. Thus for any $v\in V,$
$$d_{S}(v)=\sum_{i>0}(w-i-1)y_v(i)\equiv 0\mod (w-1).$$
That is, for any  vertex $v\in V$, there exists some integer $k_v\geq 0$ such that
$$\sum_{i>0}iy_v(i)=k_v(w-1).$$
The integer $k_v\geq 1$ if and only if $v$ is not isolated in $S$.
Sum both sides above over all vertices $v\in V$ to get
\begin{equation}\label{equation41}
\sum_{v\in V}\sum_{i>0}iy_v(i)=\sum_{v\in V}k_v(w-1)\ge s(w-1).
\end{equation}
On the other hand,
\begin{equation}\label{equation4}
\sum_{v\in V}\sum_{i>0}iy_v(i)=\sum_{i>0}i\sum_{v\in V}y_v(i)=\sum_{i>0}iy(i)(w-i).
\end{equation}
Combine (\ref{equation3}), (\ref{equation41}) and (\ref{equation4}),
$$\sum_{i>0}iy(i)(w-i)\ge s(w-1)\ge \sum_{i>0}iy(i)(w-1),$$
which holds if and only if  $y(i)=0$ for all $i>1$. This contradicts to Lemma \ref{lemmaforupperboundnonexisting}.
\end{proof}

Since a balanced $(n, 2w-2, w)_3$ code of size $B(n)+n$ does not exist, we next construct an optimal code that is not balanced.
 Let $n=h(w-1)+1$ and $2\ell=(2k+1)(w-1)$ for some $k\in[0, \frac{w-3}2]$. We only allow codewords of type $1^{w}$ and  $1^{w-2}2^1$. By doing operation (O1) $k$ times and consuming $k(w-1)$ more edges, there are finally $\frac{w-1}2$ edges left uncovered. In this case, $x=B(n)+k,$ $y=n-k$ and $z=0.$

% If there exists such a code $\mathcal C$, let $E$ denote the set of all edges of $K_n$ uncovered by $\mathcal C$. Then
%$$\frac{n(n-1)}2=|e(K_n)|=|E|+x\frac{w(w-1)}{2}+y\frac{(w-1)(w-2)}2,$$
%Since $n\equiv 1\mod (w-1)$ and $w$ is odd, $|E|$ must be a nonzero multiple of $\frac{w-1}2.$ The following construction gives an $(n, 2w-2, w)_3$ code of size $B(n)+n$ with $|E|=\frac{w-1}2$ when $n$ is large enough.

We denote all vertices of $K_n$ as the union of three parts: one single vertex $\infty$, $k$ vertices $c_j, j\in[k]$, and  $\bbZ_{n-k-1}.$ Define a sequence  $(r_1, r_2, \ldots, r_{w-1})$  of $w-1$ different vertices by
$$
r_i=\left\{
    \begin{array}{ll}
    i-1, & i\le w-k-2;\\
    \infty, & i=w-k-1;\\
    c_{w-i}, & i\ge w-k.
    \end{array}
\right.
$$

Let $m=2w^3+1$, then by Lemma \ref{golomb} there exists a $[w-1]$-free $(m, w-1)$ modular Golomb ruler $A=\{0=a_1, \ldots, a_{w-1}\}$ on $\bbZ_m$. Assume that $a_1<a_2<\cdots<a_{w-1}<m$. When $h$ is large enough such that $h>5m$, we can embed $\bbZ_m$ into $\bbZ_{n-k-1}$ in a natural way so that $A$ is also a $[w-1]$-free  $(n-k-1, w-1)$ modular Golomb ruler in $\bbZ_{n-k-1}$. The following calculations are all in $\bbZ_{n-k-1}$.

%For any two positive integers $n<m$, the natural set embedding from $[0, n-1]$ to $[0, m-1]:$ $i\to i$ gives a set embedding $\psi$ from $\bbZ_n$ to $\bbZ_m$. We can regard any subset of $\bbZ_n$ as a subset of $\bbZ_m$.
Let
$\bm a_i=\{(a_1+i)_2,(a_2+i)_1, (a_3+i)_1, \ldots, (a_{w-1}+i)_1\} \text{ for all } i\in\bbZ_{n-k-1}.$
Let $\mathcal S'=\{\bm a_i: i\in\bbZ_{n-k-1}\}$. Replace the codeword $\bm a_{2m}\in \mathcal S'$ by two other special ones of type $1^{w-2}2^1$,
$$\bm s_1=\{(a_1+2m)_2,(a_2+2m)_1, \ldots, (a_{\frac{w-1}2}+2m)_1\}\cup\{(r_1)_1, (r_2)_1, \ldots, (r_{\frac{w-1} 2})_1\},$$
$$\bm s_2=\{(r_{\frac{w+1} 2})_1, \ldots, (r_{w-2})_1,(r_{w-1})_2\} \cup \{(a_{\frac{w+1}2}+2m)_1, \ldots,  (a_{w-1}+2m)_1\}.$$

%
%$$\bm s_1=\{(a_1')_2\}\cup\{(a_2')_1, \ldots, (a_{\frac{w-1}2}')_1\}\cup\{(r_1)_1, (r_2)_1, \ldots, (r_{\frac{w-1} 2})_1\};$$
%$$\bm s_2=\{(r_{\frac{w+1} 2})_1, \ldots, (r_{w-2})_1\}\cup\{(r_{w-1})_2\} \cup \{(a_{\frac{w+1}2}')_1, \ldots, (a_{w-2}')_1, (a_{w-1}')_1\}.$$

Let $\mathcal S=(\mathcal S'\backslash \{\bm a_{2m}\}) \cup \{\bm s_1, \bm s_2\},$ which is a set of $n-k$ cliques  of type $1^{w-2}2^1$.  Define a set of edges $E=\{e_i: i\in [\frac{w-1}2]\}$ where $e_i=r_ir_{w-i}$. It is easy to see that $e_i$ is not contained in any clique of $\mathcal S$. Then let $S=\left(\cup_{\bm a\in \mathcal S}K_{\bm a}\right)\cup E$ and $G_n=K_n-S$. We will show that $\mathcal S$ forms a packing of $K_n$, and $G_n$ has a $K_w$-decomposition when $n$ is large enough.

\begin{lemma}
The set $\mathcal S$ above forms a packing of $K_n$.
\end{lemma}

\begin{proof}
Suppose not, then there are two different codewords $\bm a,\bm b\in\mathcal S$, such that $K_{\bm a}$ and $K_{\bm b}$ share at least one edge.
 Since $\mathcal S'$ is a packing, $\bm a$ and $\bm b$ cannot both belong to $\mathcal S'$.  Furthermore, since $h>5m$ and $2m\le a_i+2m\le 3m< n-k-1$ for each $i\in [w-1],$ the $2(w-1)$ vertices $r_i$, $i\in[w-1]$ and $a_j+2m$, $j\in [w-1]$ are all distinct, that is to say $\{\bm a, \bm b\}\neq \{\bm s_1, \bm s_2\}$.

Suppose that $\bm a\in \mathcal S'\backslash\{\bm a_{2m}\}$ and $\bm b\in \{ \bm s_1, \bm s_2\}$. Since  $supp(\bm a)$ is a $[w-1]$-free $(n-k-1, w-1)$ modular Golomb ruler and each $a_i$, $i\in[w-1]$ is less than $m$ while $h>5m$. The cliques $\bm s_i$, $i=1,2$ never share any edge with $\bm a_j$ for $j\neq 2m$. So $\mathcal S$ forms a packing of $K_n.$
\end{proof}

\begin{lemma}
  There exists an integer $N=N(w)$ such that when $h>N(w)$, $G_n$ has a $K_w$-decomposition.
\end{lemma}

\begin{proof}
If all conditions given in Remark~\ref{alonremark} hold, the proof is complete.
  For any vertex $v$, if $v=r_i$ for $i\le w-k-2$, it is contained in $w$ cliques of $\mathcal S$ and one edge in $E$. So $d_{G_n}(v)=n-1-w(w-2)-1 \equiv 0\mod (w-1).$ If $v= r_i$ for $i>w-k-2$, $v$ is in one clique of $\mathcal S$ and one edge in $E$. So $d_{G_n}(v)=n-1-(w-2)-1\equiv 0\mod (w-1).$ If $v\in V(K_n)\backslash\{r_i\}_{i\in[w-1]}$, it is contained in $w-1$ cliques in $\mathcal S$, and $d_{G_n}(v)=n-1-(w-1)(w-2)\equiv 0 \mod (w-1).$ Furthermore, for each vertex $v\in [n]$, $d_{G_n}(v)\ge n-1-2w^2.$

We check that $\frac{e(G_n)}{e(K_w)}$ is an integer.
$$\begin{aligned}\frac{e(G_n)}{e(K_w)} &=\frac{n(n-1)-(n-k)(w-1)(w-2)-(w-1)}{w(w-1)}\\
 &=B(n)+\frac{(2k+1)(w-1)+k(w-1)(w-2)-(w-1)}{w(w-1)}\\
 &= B(n)+k.
\end{aligned}$$

By Remark \ref{alonremark}, there exists an integer $N=N(w)$ such that when $h>N(w)$, the $K_w$-decomposition $\P$ of $G$ exists.
\end{proof}

%After a packing $\mathcal S$ is given, we need some extra edges to form $S$.
%Let $e_i=a_ia_{w-i}$ for $i\in[\frac{w-1}2]$ and $E=\{e_i\}_{i\in [\frac{w-1}2]}$. By our construction, $e_i$ is not contained in any clique of $\mathcal S$.
%
%Let $S=\mathcal S\cup E$ and $G_n=K_n-S$. We will show that $G_n$ has a $K_w$-decomposition when $n$ is large enough by checking all conditions in Remark \ref{alonremark}.

%For any vertex $v$, if $v=a_i$ for $i\le w-k-2$, it is contained in $w$ cliques of $\mathcal S$ and one edge in $E$. So $d_{G_n}(v)=n-1-w(w-2)-1 \equiv 0\mod (w-1).$ If $v\in a_i$ for $i>w-k-2$, $v$ is in one clique of $\mathcal S$ and one edge in $E$. So $d_{G_n}(v)=n-1-(w-2)-1\equiv 0\mod (w-1).$ If $v\in [n]\backslash\{a_i\}_{i\in[w-1]}$, it is contained in $w-1$ cliques in $\mathcal S$, and $d_{G_n}(v)=n-1-(w-1)(w-2)\equiv 0 \mod (w-1).$ Furthermore, for each vertex $v\in [n]$, $d_{G_n}(v)\ge n-1-2w^2.$
%
%We check that $\frac{e(G_n)}{e(K_w)}$ is an integer.
%$$\begin{aligned}\frac{e(G_n)}{e(K_w)} &=\frac{n(n-1)-(n-k)(w-1)(w-2)-(w-1)}{w(w-1)}\\
% &=B(n)+\frac{(2k+1)(w-1)+k(w-1)(w-2)-(w-1)}{w(w-1)}\\
% &= B(n)+k.
%\end{aligned}$$
%
%By Remark \ref{alonremark}, there exists an integer $N=N(w)$ such that when $y>N(w)$ the $K_w$-decomposition $P$ of $G$ exists.

We color all cliques in $\P$ as type $1^w$ and get a code $\mathcal P$ of $B(n)+k$ codewords. By Corollary \ref{transfer}, $\mathcal C=\mathcal S\cup\mathcal P$ is an optimal $(n, 2w-2, w)_3$ code with size $B(n) +n$ but not balanced.
Combining the result above and Subsection \ref{seciii:a}, we have proven the following theorem.

\begin{theorem}\label{theoremformod1case}
For any integer $w\ge 5$, there exists an integer $N=N(w)$ such that for any $n>N$ and $n\equiv 1 \mod (w-1),$ $A_3(n, 2w-2, w)=B(n)+n$.
\end{theorem}

However, by Theorem \ref{upperboundnonexisting}, an optimal balanced $(n, 2w-2, w)_3$ code does not exist when $n\equiv 1\mod (w-1)$, and $n(n-1)-n(w-1)(w-2)\equiv (2k+1)(w-1) \mod w(w-1)$ for some $k\in[0, \frac{w-3}2]$.

\section{Constructions for $n\equiv t \mod (w-1)$ with $t\in[2, w-2]$}\label{Constructionforgeneralt}

In this section, we show that an optimal  balanced $(n, 2w-2, w)_3$ code exists for sufficiently large $n\equiv t \mod (w-1)$ when $t\in[2, w-2]$. Our method is similar to the cases when $t=0,1$, that is, constructing a code $\S$ satisfying Lemma~\ref{condofS}, but is more complicated.
Our main result is the following theorem.

%We have solved the cases completely when $n\equiv 0,1 \mod (w-1)$ for any given $w\ge 5$ and sufficiently large $n$ by split them into four cases in Sections \ref{mainconsturctionfor01}--\ref{additionalconstructionB}. Constructions for the former three cases are balanced codes, in which we determine some important values like $a, b, c, x, y$ and $z$ before starting and use them to guide our constructing process. The fourth case, however, is soled differently since no balanced code can reach the theoretical upper bound, and instead we construct an unbalanced optimal code by simple modifications of a huge round of codewords from a restricted modular Golomb ruler. In this section, we consider the cases when $n\equiv t \mod {(w-1)}$ for all $2\leq t\leq w-2$.

\begin{theorem}\label{theoremformod-t-case}
For any integers $w\ge 5$ and $t\in[2, w-2]$, there exists an integer $N=N(w)$ such that for any $n>N$ and $n\equiv t\mod (w-1),$ there is a balanced $(n, 2w-2, w)_3$ code with size $B(n)+n$, that is $A_3(n, 2w-2, w)=B(n)+n$.
\end{theorem}

\subsection{Strategy}\label{scratch}

When $w$ is fixed, let $n=h(w-1)+t$ for some $t\in [2, w-2]$ be a sufficiently large integer. With $\ell$ defined in Equation~(\ref{2l}), there are unique non-negative integers $r$ and $k$ such that $2\ell= 2k(w-1)+2r$ for $0\le 2r< 2(w-1)$ and $0\le 2k(w-1)+2r<w(w-1).$ In this sense, $r$ and $k$ are functions of $n$ with values smaller than $w.$

We apply operations (O1) $k$ times and (O2) $r$ times to digest the exceeding $\ell$ edges, that is $a=k$ and $b=r$. By Equation~(\ref{rm}), $R_m=w-t$ and $R(v)\in \{w-t,2w-t-1\}$. Then by Equation~(\ref{dsv}), $c=h(w-t)+k+t+\frac{2r-(t-1)t}{w-1}$, which is a positive integer by Equation (\ref{2l}). So the desired numbers of codewords of types $1^{w}$, $1^{w-2}2^1$ and $1^{w-4}2^2$ are $x=B(n)+k+r,$ $y=n-k-2r$ and $z=r$, respectively.

By Lemma~\ref{condofS}, it suffices to construct a code $\mathcal S$ consisting of $n-k-2r$ codewords of type $1^{w-2}2^1$  and $r$ codewords of type $1^{w-4}2^2$ satisfying all conditions in Lemma~\ref{condofS}. The vertex set of $K_n$ is defined as a union of three parts with the prescribed distribution of different values of $R(v)$ in each part as follows.
\begin{itemize}
\item[(R1)] Let $\tilde{h}=h-w(w+2)$ and $n'=\tilde{h}(w-1)$, then the first part of $V(K_n)$ is $\bbZ_{n'}$. There will be $\tilde{h}(w-t)$ vertices in $\bbZ_{n'}$ with $R(v)=w-t$, and all others with $R(v)=2w-t-1$.
\item[(R2)] Let $B$ be the second part consisting $r(w-2)$ vertices, which are $\bar{b}_i$, $i\in [2r]$ and $b_j$, $j\in [r(w-4)]$. All vertices in $B$ have $R(v)=2w-t-1$.
\item[(R3)] Let $C$ be the third part consisting all remaining vertices, which are $\bar{c}_i$, $i\in[c-\tilde{h}(w-t)]$ with $R(v)=w-t$, and $c_j$, $j\in [\kappa]$ with $R(v)=2w-t-1$. Here $\kappa=n-c-\tilde{h}(t-1)-r(w-2)$. Note that $\kappa>k$, and there will be $k$ vertices $c_j$ that are not colored by $2$ in the desired $\S$.
\end{itemize}
Note that in this setting, the number of vertices with $R(v)=w-t$ is exactly $c$, and with $R(v)=2w-t-1$ is $n-c$. See Fig.~\ref{picture} for an illustration of the vertices and $R(v)$.
Our strategy of constructing $\S$ is as follows.
\begin{figure}[!htbp]
\centering
\includegraphics[height=45mm,width=70mm]{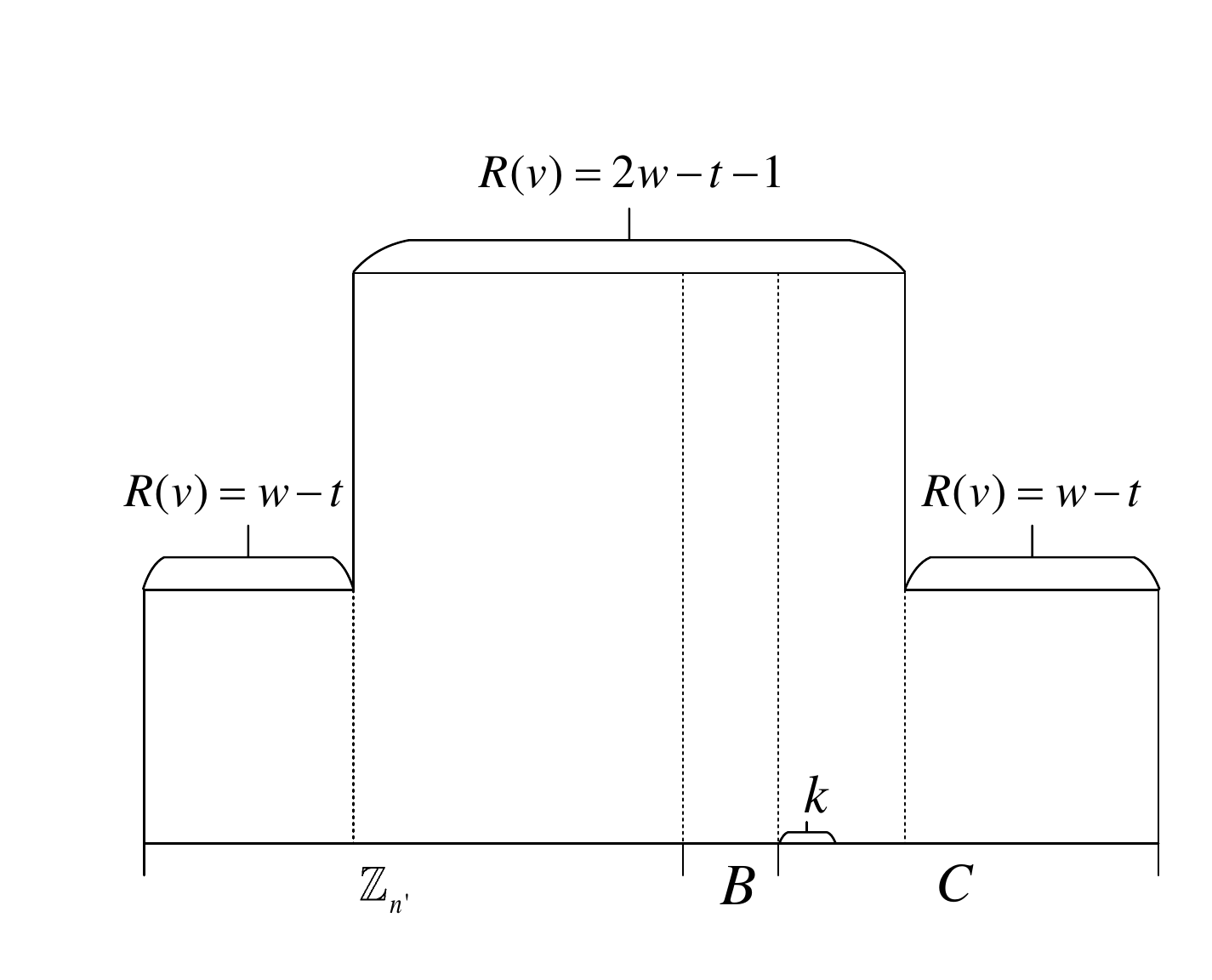}
\caption{Three parts of $V(K_n)$ with the corresponding $R(v)$}
\label{picture}
\end{figure}
%\begin{itemize}
%\item[(S1)]

First, we construct an $(n',2w-2,w)_3$ code $\H$ over  $\bbZ_{n'}$ consisting of $n'$ codewords of type $1^{w-2}2^1$, such that the requirement of $R(v)$ in (R1) is satisfied. See  Construction~\ref{doubleGolomb} and Proposition~\ref{packingdoubleGolomb} for details of $\H$. View $\H$ as a code of length $n$ over the vertex set of $K_n$. Note that $\H$ colors each vertex in $\bbZ_{n'}$  by $2$ exactly once.

Second, we construct  $r$  codewords of type $1^{w-4}2^2$
    $$\bm z_i=\{(\bar{b}_{2i-1})_2, (\bar{b}_{2i})_2\}\cup\{(b_{(w-4)(i-1)+1})_1, (b_{(w-4)(i-1)+2})_1, \ldots, (b_{(w-4)i})_1\},\text{ for all } i\in\left[r\right],$$ whose supports form a partition of the set $B$.
    Let $\mathcal Z=\{\bm z_i: i\in[r]\}.$ Each $\bar{b}_i, i\in [2r]$ is colored by $2$ exactly once.

Third, construct an $m\times (w-1)$ array $M$ with entries in $B\cup C$, where $m=r(w-4)+|C|-k=n-n'-2r-k$, such that
    \begin{itemize}
    \item the first column of $M$ covers each vertex $b_j$, $j\in[r(w-4)]$, and each vertex in $C$ except some $k$ vertices $c_j$ exactly once.
    \item the number of occurrence of each $v$ in $M$ is $R(v)-2$ if $v\in B$, and is $R(v)$ if $v\in C$, where  $R(v)$ is defined in (R2) and (R3).
    \end{itemize}
Such an array $M$ exists by simply checking the equality $m(w-1)=\sum_{v\in B}(R(v)-2)+\sum_{v\in C}R(v)$. Each row of $M$ will serve as a candidate of a codeword of type $1^{w-2}2^1$ by coloring the first element by $2$ and all others by $1$. However, the current $M$ is not valid since each row may have repeated elements. Let $\M$ denote the collection of colored rows of the current $M$. We would like to do some operations of  $\M\sqcup \H$, so that finally, $\M\sqcup \H\sqcup \mathcal Z$ will give the desired $\S$. The symbol $\sqcup$ means the disjoint union of collections.

Before presenting our algorithms of the operations of  $\M\sqcup \H$ in Section\ref{constructingmathcalS}, we give a construction of $\H$.

\subsection{Construction of $\H$ and its Properties}

\begin{construction}\label{doubleGolomb}
Suppose that  $\tilde{h}>2w^3$. Start from an $(\tilde{h}, w-1)$ modular Golomb ruler $\{a_0=0, a_1, a_2, \ldots, a_{w-2}\}$, and view it as a subset in $\bbZ_{n'}$, where $n'=\tilde{h}(w-1)$.  The  code $\mathcal H$ consisting of $n'$ codewords of type $1^{w-2}2^1$ is constructed below.

\begin{itemize}
  \item[(1)] For each $i\in [0, t-2]$ and $j\in [0, \tilde{h}-1],$ let \[\bm a_{i, j}=\{(j(w-1)+i)_1, ((a_1+j)(w-1)+i)_1, \ldots, ((a_{w-3}+j)(w-1)+i)_1, ((a_{w-2}+j)(w-1)+i)_2\}.\]

  \item[(2)] For each $i\in[0, w-t-1]$ and $j\in [0, \tilde{h}-1]$, let \[supp(\bm b_{i, j})=\{(i+j)(w-1), (2i+j)(w-1)+1, (3i+j)(w-1)+2,\ldots,  ((w-1)i+j)(w-1)+w-2\},\] where the vertex $((t+i)i+j)(w-1)+(t+i-1)$ is colored by $2$ and all other vertices are colored by $1$.

\end{itemize}

Then $\mathcal H=\{\bm a_{i, j}: i\in [0, t-2], j\in [0, \tilde{h}-1]\}\cup\{\bm b_{i, j}: i\in [0, w-t-1], j\in [0, \tilde{h}-1]\}.$
It is easy to verify that each vertex in $ \bbZ_{n'}$  is colored by 2 exactly once in $\mathcal H$.
\end{construction}

\begin{proposition}\label{packingdoubleGolomb} The code $\mathcal H$ in Construction~\ref{doubleGolomb} is an $(n', 2w-2, w)_3$ code.
Further, there are exactly $\tilde{h}(w-t)$ vertices in $\bbZ_{n'}$ with $R(v)=w-t$, and all others with $R(v)=2w-t-1$, which satisfies the requirement of $R(v)$ in (R1).
\end{proposition}

\begin{proof}
For any vertex $v\in \mathbb Z_{n'}$, there are exactly $w-t$ cliques  $\bm  b_{i, j}$ containing $v$. Let $v \equiv i \mod (w-1)$ for some $i\in[0, w-2]$. If  $i$ is in $[0, t-2],$ then $v$ is contained in exactly $w-1$ cliques  $\bm a_{i, j}$, otherwise $v$ does not appear in any clique of $\bm a_{i, j}.$ This proves the distribution of $R(v)$.

It is left to prove that the collection of all supports of codewords in $\mathcal C$ forms a packing.
 Here are five cases.
\begin{itemize}
\item[(1)] By modular Golomb ruler, for any $0\le j\ne j'\le \tilde{h}-1$ and $i\in [0, t-2],$ $|supp(\bm a_{i, j})\cap supp(\bm a_{i, j'})|\le 1.$

\item[(2)] Since any vertex of $supp(\bm a_{i, j})$ modulo $(w-1)$ equals $i$, then for any $0\le i_1\ne i_2\le t-2$ with any $j$ and $j'$, $|supp(\bm a_{i_1, j})\cap supp(\bm a_{i_2, j'})|=0.$

\item[(3)] It is easy to see that $|supp(\bm b_{i, j})\cap supp(\bm b_{i, j'})|=0$ for any $0\le j\ne j' \le \tilde{h}-1.$

\item[(4)] For any $0\le i_1<i_2\le w-t-1$ and $0\le j, j'\le \tilde{h}-1$, we also have $|supp(\bm b_{i_1, j})\cap supp(\bm b_{i_2, j'})|\le 1.$ Otherwise, there exists a pair $(s_1, s_2)$ from $[w-1]$ with $s_1<s_2$ such that the following two equations hold:
$$(s_1i_1+j)(w-1)+s_1-1\equiv (s_1i_2+j')(w-1)+s_1-1 \mod n';$$
$$(s_2i_1+j)(w-1)+s_2-1\equiv (s_2i_2+j')(w-1)+s_2-1 \mod n'.$$
    We calculate the differences of these two equations on both side and get
    $(s_2-s_1)(i_2-i_1)\equiv 0 \mod \tilde{h}.$
    However, $0\le (s_2-s_1)(i_2-i_1)\le (w-2)(w-t-1)<2w^2<\tilde{h}$, a contradiction.

\item[(5)] For any $0\le i\le t-2$, $0\le i'\le w-t-1$ and $0\le j, j'\le h-1$, we have $|\bm a_{i, j}\cap \bm b_{i', j'}|\le 1.$ That is because all vertices in $supp(\bm a_{i, j})$ are the same while all vertices of $supp(\bm b_{i', j'})$ are pairwise different after modulo $(w-1).$
\end{itemize}
\end{proof}

\subsection{Constructing $\mathcal S$}\label{constructingmathcalS}

Let $\H$, $\Z$, $\M$ and $M$ be defined as before. By abuse of notation, we also use $\Z$, $\M$, and $\H$ to denote the corresponding collections of unlabeled (multi) subsets of $V(K_n)$. We express $\H$ as row vectors of an $n'\times (w-1)$ array $H$ with each row being the support of a codeword in $\H$, where the only vertex labeled by $2$ is always located in the first column. Let $\bar{{M}}$, $\bar{H}$ be obtained from $M$ and $H$ by deleting the first column.

%The sum operation between multisets are denoted as symbol ``$+$" in our following texture.
Let $\S'$ be the disjoint union of the collections $\H$, $\Z$ and $\M$, denoted by $\S'=\H\sqcup \Z\sqcup\M$. Observe that if we can make  $\S'$ an $(n,2w-2,w)_3$ code by only exchanging entries in  $\bar{{M}}$ and $\bar{H}$, then the resultant code will satisfy all conditions of Lemma~\ref{condofS}, that is, contain $b=r$ codewords of type $1^{w-4}2^2$ whose supports are pairwise disjoint (codewords in $\Z$), contain $n-a-2b=n-k-2r$ codewords of type $1^{w-2}2^1$ (codewords in modified $\H\sqcup \M$), and has exactly $c$ vertices with $R(v)=R_m=w-t$ and all others with $R(v)=R_m+w-1$. Since each vertex is colored by $2$ in $\S'$ exactly once except $k$ vertices $c_j$, by  Corollary~\ref{transfer}, it suffices to exchange entries in  $\bar{{M}}$ and $\bar{H}$ so that there are no repeated entries in each member of $\M\sqcup\H$ and no common pairs among all members of $\H\sqcup \Z\sqcup \M$. The initial state of $\H\sqcup \Z$ is already an $(n,2w-2,w)_3$ code, so we can do this exchange one by one on all entries of $\bar{{M}}$ with a natural order.

In the algorithm to follow, all our operations are supposed to exchange some entry in $\bar{{M}}$ and some entry in $\bar{H}$ only. Let $\bm m_i$ denote the $i$th row of $M$, and let $m_{i,j}$ be the $(i,j)$th entry of $M$. Suppose that we have modified all entries $m_{i',j'}$ with $1\leq i'< i$ for any $2\leq j'\leq w-1$, and $m_{i,j'}$ for  $2\leq j'<j$, such that
\begin{itemize}
\item[(h1)] the updated $\H\sqcup\{\bm m_{i'}: i'<i \}$ has no repeated vertex in each element;
\item[(h2)] the updated entries $m_{i, 1}, m_{i, 2}, \ldots, m_{i, j-1}$ are pairwise distinct;
\item[(h3)] the updated multiset $\H\sqcup \Z \sqcup\{\bm m_{i'}: i'<i \} \sqcup \{\{m_{i,j'}: j'<j\}\}$ is a set, and further a packing over $V(K_n)$.
\end{itemize}
Now we focus on the entry $m_{i,j}$. If the following properties (p1)-(p3) are already satisfied, we do nothing but change $j\rightarrow j+1$ if $j< w-1$, or change $i\rightarrow i+1$ and $j\rightarrow 2$ if $j=w-1$.  Otherwise, we need to find an entry $\hbar$ in the current $\bar{H}$, so that exchanging $m_{i,j}$ and $\hbar$ will result in that
\begin{itemize}
\item[(p1)] the updated $\H\sqcup\{\bm m_{i'}: i'<i \}$ have no repeated vertex in each element;
\item[(p2)] the updated entries $m_{i, 1}, m_{i, 2}, \ldots, m_{i, j}$ are pairwise distinct;
\item[(p3)] the updated multiset $\H\sqcup \Z \sqcup\{\bm m_{i'}: i'<i \} \sqcup \{\{m_{i,j'}: j'\le j\}\}$ is a set, and further a packing over $V(K_n)$.
\end{itemize}
The properties (p1)-(p3) will be used as the new (h1)-(h3) in the next step of modification. We claim that the existence of $\hbar$ is always possible for fixed $w$ and sufficiently large $n$. We give the proof below.

The number $m$ of rows of $M$ is $m=t+w(w+2)(w-1)-2r-k\leq 2w^3$, so the set $T$ of distinct entries in any updated $M$ is of size at most $2w^4$. Each vertex of $T$ can appear in at most $R_m+w-1=2w-t-1$ members of $\H\sqcup \Z\sqcup\M$, so the set $N$ of vertices in $V(K_n)$ appearing in the same member of  $\H\sqcup \Z\sqcup \M$  with some vertex in $T$ is of size at most $2w^4\cdot(2w-t-1)\cdot(w-1)\leq 4w^6$. All vertices in $N$  appear in at most $4w^6\cdot(2w-t-1)\leq 8w^7$ members of $\H\sqcup \Z\sqcup \M$. Since $8w^7=O(1)$ when $n$ goes to infinity, $|\H|=\Theta(n)$, there is at least one codeword of $\H$ containing no vertex in $N$. Choose such a codeword $\bm a\in \H$ and a vertex $\hbar\in supp(\bm a)$ that is colored by $1$. It is clear that $\hbar$ is an entry of $\bar{H}$. Let $\bm a$ be modified by replacing $\hbar$ with $m_{i,j}$, and update $\H$. Replace $m_{i,j}$ by $\hbar$ and update $M$.  We show that the three properties (p1)--(p3) follow after this round of updating.

In fact,
by the disjointness between $supp(\bm a)$ and $N$, $supp(\bm a)$ contains no vertices from $M$. Thus properties (p1) and (p2), and the part of telling the multiset is a set in property (p3) are obvious. The only thing need to prove is the packing property in (p3). For convenience, let $\mathcal C$ be the collection $\H\sqcup \Z \sqcup\{\bm m_{i'}: i'<i \} \sqcup \{\{m_{i,j'}: j'<j\}\}$ before exchanging $m_{i,j}$ and $\hbar$, and let $\mathcal {\bar C}$ be  $\H\sqcup \Z \sqcup\{\bm m_{i'}: i'<i \} \sqcup \{\{m_{i,j'}: j'\le j\}\}$ after the exchange. If (p3) is false, there exist two vertices $u$ and $v$ in $V(K_n)$ such that at least two different elements, denoted as $\bm u$ and $\bm v$, in $\mathcal {\bar C}$ contain both of them. If $m_{i, j}$ and $\hbar$ are both out of $\{u, v\}$, the number of members in $\mathcal {\bar C}$ containing both $u$ and $v$ is the same as that in $\mathcal C$. By induction hypothesis $\mathcal C$ is already a packing and thus such two elements $\bm u$ and $\bm v$ do not exist. As for the case when $\{u, v\}$ contains at least one vertex of $m_{i, j}$ and $\hbar$, contradiction can be derived directly by that $supp(\bm a)$ is disjoint from $N$.

Finally, we complete the proof by showing the initial hypothesis holds. When $i=1$ and $j=2$, no entries of $M$ need to be modified before. So properties (h1)-(h3) are satisfied naturally from our constructions of $\H$ and $\Z$.  Thus we can find an $\hbar$ to exchange $m_{1,2}$ such that (p1)-(p3) are satisfied, which will be used as the new (h1)-(h3) when we modify $m_{1, 3}$. Continuing this exchanging algorithm until we modify all entries in $\bar{{M}}$, so that the modified $\H\sqcup \Z\sqcup \M$ is a packing over $V(K_n)$. Then $\S$ is the set $\H\cup \Z\cup \M$ with the prescribed color structures.

%This proof does not work well in the cases when $t=0$ or $t=1.$ Because in our construction we need to divide $[n]$ into $T_1$ and $T_2$ to define $A$ and $B$. Since $|T_1|=n\frac{t-1}{w-1}-\frac{k(w-1)+2r}{w-1},$ when $t\le 1$ the size of $T_1$ is negative for large enough $n$ which is impossible.

\section{Conclusion}\label{conc.}

In this paper we study ternary constant weight codes in $\ell_1$-metric for any given length $n$, weight $w$ and minimum distance $2w-2.$ We show that the upper bound on the maximum size of an $(n, 2w-2, w)_3$ code in \cite{chen2020optimal} can be achieved for all sufficiently large $n$. Our main result is stated below by combining all results in Theorems \ref{theoremformod0case},   \ref{theoremformod1case} and \ref{theoremformod-t-case}.

\begin{theorem}\label{theoremformod-all-case}
For any integer $w\ge 5$, there exists an integer $N=N(w)$ such that for any $n>N$, $A_3(n, 2w-2, w)=B(n)+n$.
Further, there exists an optimal $(n, 2w-2, w)_3$ code that is  balanced if and only if the following does not happen.
\begin{itemize}
\item The weight $w$ is odd, $n\equiv1\mod (w-1)$ and $n(n-1)-n(w-1)(w-2)\equiv (2k+1)(w-1)\mod w(w-1)$ for some integer $k$ such that $2k+1\in[0, w-1].$
\end{itemize}
\end{theorem}

 The optimal codes that we construct only contain codewords of types $1^w$, $1^{w-2}2^1$ and $1^{w-4}2^2$.
The method we use is  transferring the problem of finding an optimal $(n, 2w-2, w)_3$ code to the problem of finding  a packing of $K_n$ of the same size by colored cliques with extra restrictions. We use $B$-free modular Golomb rulers to find a packing $\mathcal S$ with size no more than $n$ of colored cliques to produce all codewords of types $1^{w-2}2^1$ and $1^{w-4}2^2$. Our constructions of $\S$ satisfy the conditions of  Lemma~\ref{condofS}, so that $K_n\setminus (\cup_{\bm a\in \S}K_{\bm a})$ has a $K_w$-decomposition by Theorem \ref{alon} when $n$ is large enough, which produces all codewords of type $1^w$.

Our method works well on the distance $2w-2$. However, it does not work when the minimum distance $d<2w-2.$ Meanwhile, since the existence of a $K_w$-decomposition in Theorem \ref{alon} is not explicit, efficient algorithms to find out an optimal $(n, 2w-2, w-1)_3$ code in polynomial time are in need. We leave these problems for future study.

\vskip 10pt
\bibliographystyle{IEEEtran}
\bibliography{cwc}
\end{document}